\documentclass[12pt,a4paper]{article}
\usepackage[left=1.2in,right=1.2in,top=.8in,bottom=.8in]{geometry}
\usepackage{amsmath,amsfonts,amsthm,amssymb,graphicx,xcolor}  
\newtheorem{theorem}{Theorem}
\newtheorem*{theorem*}{Theorem}
\newtheorem*{proposition*}{Proposition}
\newtheorem{lemma}{Lemma}
\newtheorem{corollary}{Corollary}
\newtheorem{definition}{Definition}
\newtheorem{proposition}{Proposition}
\theoremstyle{remark}
\newtheorem{remark}{Remark}
\usepackage[utf8]{inputenc}
\usepackage{color}
\usepackage{hyperref}
\usepackage[mathscr]{euscript}
\usepackage{enumerate}
\usepackage{mathtools}
\mathtoolsset{showonlyrefs=true}

\newcommand{\e}{\mathrm{e}}

\renewcommand{\d}{\mathrm{d}}

\newcommand{\R}{\mathbb{R}}

\newcommand{\E}{\mathcal{E}}

\newcommand{\eps}{\varepsilon}

\newcommand{\N}{\mathbb{N}}
\newcommand{\f}[1]{\footnote{\tbl{#1}}}
\newcommand{\s}{\hspace{0.5pt}}
\newcommand{\supp}{\mathop{\rm supp}}

\newcommand{\tildeR}{\widetilde{\mathcal{R}}}
\newcommand{\dualH}{\widetilde{H}}
\newcommand{\F}{\mathcal{F}}		

\renewcommand{\;}{\s ;\s }
\newcommand{\ccdot}{\,\cdot\,}

\newcommand{\p}{\partial}
\newcommand{\tbl}[1]{\textcolor{blue}{#1}}

\newcommand{\norm}[1]{\lVert #1 \rVert}
\newcommand{\abs}[1]{\lvert #1 \rvert}

\usepackage{authblk}
\begin{document}
\title{Uniqueness and stability of an inverse problem for a semi-linear wave equation}

\author[1]{Matti Lassas}

\author[2]{Tony Liimatainen}

\author[3]{Leyter Potenciano-Machado}

\author[4]{Teemu Tyni}

\affil[1,4]{Department of Mathematics and Statistics, University of Helsinki, Helsinki, Finland}
\affil[2,3]{Department of Mathematics and Statistics, University of Jyv\"askyl\"a, Jyv\"askyl\"a, Finland}



\maketitle

\begin{abstract}

We consider 
the recovery of a potential associated with a semi-linear wave equation on $\R^{n+1}$, $n\geq 1$.
%
We show that an unknown potential $a(x,t)$, supported in $\Omega\times[t_1,t_2]$, of the wave equation $\square u +a u^m=0$ can be recovered in a H\"older stable way from the map $u|_{\p \Omega\times [0,T]}\mapsto 
\langle\psi,\p_\nu u|_{\p \Omega\times [0,T]}\rangle_{L^2(\p \Omega\times [0,T])}$. This data is equivalent to the inner product of the Dirichlet-to-Neumann map with a measurement function $\psi$.
  We also prove similar stability result for the recovery of $a$ when there is noise added to the boundary data. The method we use is constructive and it is based on the higher order linearization. As a consequence, we also get a uniqueness result. We also give a detailed presentation of the forward problem for the equation $\square u +a u^m=0$.

\end{abstract}

In this paper we study an inverse boundary value problem for a non-linear wave equation. The inverse problems we study are the uniqueness and stability of recovering an unknown potential $a\in C_c^\infty(\Omega\times \R)$ of the non-linear wave equation
\begin{equation}\label{eq:intro_wave-eq}
\begin{cases}
\square\s u(x,t) + a(x,t)\s u(x,t)^{m}   = 0 &\text{in } \Omega\times[0,T],\\
\qquad\qquad\qquad\qquad\qquad \ \s u=f &\text{on } \p\Omega\times[0,T], \\
\quad \ \ \ \ \, \s u\big|_{t=0} = 0,\quad \partial_t u\big|_{t=0} = 0 &\text{on } \Omega
\end{cases}
\end{equation}
from the Dirichlet-to-Neumann map (DN map) of the equation.  Here $m\geq 2$ is an integer, $\Omega$ is an open subset of $\R^n$, $T>0$, and $\square$ is the standard wave operator $\p_t^2-\Delta$ in $\R^{n+1}$. We assume 
that the potential $a=a(x,t)$ can depend on the time variable $t$. 
%


Inverse problems for Equation~\eqref{eq:intro_wave-eq} are natural counterparts to the widely studied inverse problems for the linear operator $\square u + a\s u$. We refer to~\cite{KKL01} for inverse problems for linear wave equations. Equations of the type~\eqref{eq:intro_wave-eq} arise for example in quantum mechanics in the context of the Klein-Gordon equation.

We will show that the boundary value problem for Equation~\eqref{eq:intro_wave-eq} has a unique small solution $u$ for sufficiently small boundary data $f\in H^{s+1}(\p \Omega \times [0,T])$, where $s\in \N$ and $s+1>(n+1)/2$. Precisely this means that there is $\eps>0$ such that whenever $\norm{f}_{H^{s+1}(\p \Omega \times [0,T])}\leq \eps$ , there is a unique solution $u_f$ to \eqref{eq:intro_wave-eq} with sufficiently small norm in the \emph{energy space} $E^{s+1}$
\[
 E^{s+1}=\bigcap_{0\leq k \leq s+1}C^k([0,T]\; H^{s+1-k}(\Omega)).
\]
Here $H^{s+1}$ is the standard Sobolev space.  We will call $u_f$ the unique small solution. 
The DN map $\Lambda$ is then defined by using the unique small solution by the usual assignment,
	\begin{align*}
	\Lambda: H^{s+1}(\p \Omega \times [0,T])\to H^{s}(\p \Omega \times [0,T]), \ \  f \mapsto \p_{\nu} u_f|_{\p \Omega \times [0,T]}.
	\end{align*}
	Here $\p_{\nu}$ denotes the normal derivative on the boundary $\p \Omega \times [0,T]$. See Section \ref{sec:forward-problem} for more details on well-posedness.

Let us briefly mention some results on inverse problems for linear equations. In the case where the underlying equation is linear {and elliptic}, a standard example is from the {pioneering work} of Calder\'on \cite{Ca80}, known nowadays as Calder\'on's inverse problem.
This problem was solved in  the fundamental papers by  Sylvester and Uhlmann~\cite{SyUh87}, in the three and higher dimensional case,
and Nachman~\cite{N96}  and Astala and P\"aiv\"arinta \cite{AP}, in the two dimensional case.
For a gentle introduction to Calder{\'o}n's problem and related topics, see for instance \cite{GT13,KS14,Uh13} and the references therein. 
Numerical techniques for the problem
are discussed in \cite{Siltanen2,Siltanen1}. 
 For the linear hyperbolic equation,  the results on uniqueness and their corresponding quantitative versions have been studied using Carleman estimates and the complex geometric optics, see  \cite{BK81,IY01,NSU88}.

Uniqueness results for inverse problem for the wave equation with vanishing initial data are obtained using the boundary control method, originated by Belishev and Kurylev~\cite{Be87,BeKu92},
 that combines  the wave propagation and controllability results, see also \cite{KKL01}.
The boundary control method  allows also an effective numerical algorithm \cite{deHoop}. Recent geometrical results 
on determining Riemannian manifolds with partial data or with general operators
are considered in \cite{AKKLT,Helin,Isozaki,KKLO,KrKL,KOP,Lassas,LO}.
The boundary control method has been applicable only in the cases where the coefficients of the equation
are time-independent, or when the lower order terms are real analytic in time variable \cite{Esk}.

Inverse problems for linear wave equations with lower order terms depending on the time variable have been considered 
in \cite{FIKO,St,SY}. These method apply propagation of singularities along bicharacteristics
 to determine the integrals of the coefficients along rays. These results are closely related to the methods used
 in this paper with the significant difference that in these results one has to assume that the complete Dirichlet-to-Neumann operator or a scattering operator is known.

A recent observation by Kurylev, Lassas and Uhlmann \cite{KLU18} is that a non-linearity in the studied equation can be used as a beneficial tool in a corresponding inverse problem. By exploiting the non-linearity, some still unsolved inverse problems for linear hyperbolic equations have been solved for their non-linear counterparts. For the scalar wave equation with a quadratic non-linearity, they in~\cite{KLU18} proved that local measurements determine the global topology, differentiable structure and the conformal class of the metric $g$ on a globally hyperbolic $3+1$-dimensional Lorentzian manifold. Following this observation, there has been 
a surge of interesting results for inverse problems for non-linear equations. The authors of~\cite{LUW18} studied inverse problems for general semi-linear wave equations on Lorentzian manifolds, and in \cite{LUW17} they studied analogous problem for the Einstein-Maxwell equations. Recently, inverse problems for non-linear equations using the non-linearity as a tool, have been studied in   
~\cite{AsZh17, CaNaVa19, Chen2019,CLOP2020,dH2019, dH2020,FeOk20,FO, KaNa02,  KrUh19, KrUh20,KLOU2014, LaUhYa20, LLLS19a, LLLS19b, OSSU,SunUh97, UhWa18,WZ2019}.
The works mentioned above use the so-called \emph{higher order linearization} method, which we will explain later.
In this work we continue to use the non-linearity as a tool to prove a stability estimate for the described inverse problem for Equation~\eqref{eq:intro_wave-eq}.
%
%
 Our main results are a H\"older stability estimate for the inverse problem of the Equation~\eqref{eq:intro_wave-eq} and a constructive way to approximate $a$ in the presence of additive noise.
%
%
The noise is modelled by a bounded mapping $\mathcal{E}: H^{s+1}(\Sigma) \to H^s(\Sigma)$.
We do not assume that $\mathcal{E}$ is a linear mapping.
%
The main idea is to use the non-linearity to approximate ``virtual sources" which are multiplied by the unknown potential.
%

%

	We present our main results next. We denote the \emph{lateral boundary} $\p\Omega\times[0,T]$ by
\[
 \Sigma=\p\Omega\times[0,T]
\]
or by $\Sigma^T$ if we wish to emphasize the corresponding time interval.  
Due to the finite propagation speed of solutions to the wave equation, there are natural limitations on the areas of $\R^{n+1}$ where we can obtain information in the inverse problem. Regarding this, we introduce a technical notion of \emph{admissible potentials}, which will be useful in presenting our results. 

Let us denote
\begin{align}\label{def:d}
d&:= 2\,\inf\left\{ r>0 \mid \Omega\subset B_r(x),\text{ for some } x\in\R^n \right\},
\end{align}
where $B_r(x)$ is the ball of radius $r$ centered at $x\in\R^n$.
By Jung's theorem, $\mathrm{diam}(\Omega)\leq d \leq \mathrm{diam}(\Omega) \sqrt{2n/(n+1)}$.
We assume that 
$$T\geq 2d+2\lambda$$
for some given $\lambda>0$. Let then 
\begin{equation}\label{def:t1t2}
\begin{split}
t_1 = d+\lambda \quad\text{and}\quad t_2 = T-d-\lambda.
\end{split}
\end{equation}
The parameter $\lambda$ will be used to build small neighbourhoods of the times $t=0$ and $t=T$. Its precise value is not important.
\begin{definition}[Admissible potentials]\label{def:admissible}
Given $s\geq 0$ and $L>0$, the class of admissible potentials 
is defined as the set consisting of all the functions $a\in C^{s+1}_c(\Omega\times \R)$ satisfying the following conditions
\[
\norm{a}_{C^{s+1}}\leq L,
\]
\[
\supp (a)  \Subset \Omega\times [t_1,t_2].
\]
Note that $[t_1,t_2]\subset[0,T]$. 
\end{definition}

This definition is motivated by the need to reach any point in the support of $a$ by sending waves from $\Sigma$, and that the corresponding measurements can be detected on $\Sigma$. If $a$ satisfies the conditions of the definition, we say that $a$ is \emph{admissible} with implied constants $L$ and $\lambda$.

Our first result is the following

\begin{theorem}[Uniqueness]\label{thm:uniqueness}
Let $\Omega\subset \R^n$ be a bounded domain with a smooth boundary. Let $m\geq 2$ be an integer and $s+1>(n+1)/2$.
There is a measurement function $\psi\in L^2(\Sigma)$ such that the following is true:
 Assume that $a\in C_c^{s+1}(\Omega\times \R)$ is admissible.  Let $u_f$ be the solution to~\eqref{eq:intro_wave-eq} for small enough $f\in H^{s+1}(\Sigma)$.
 
Then the real-valued non-linear map
\begin{equation}
 \lambda_{\psi}: f\mapsto \langle \psi,\p_\nu u_f\rangle_{L^2(\Sigma)}
\end{equation}
determines $a(x,t)$ uniquely.
\end{theorem}

The \emph{measurement function } $\psi\in L^2(\Sigma)$ appearing in the statement of Theorem~\ref{thm:uniqueness} is the restriction of a solution of the following backwards wave equation to the lateral boundary:
\begin{equation}\label{eq:defv_0}
\begin{cases}
\square v_0 = 0, &\text{in } \R^n\times[0,T],\\
v_0 \equiv 1,&\text{in } \Omega\times[t_1,t_2],\\
v_0\big|_{t=T} = \p_t v_0\big|_{t=T} = 0, &\text{in }\Omega,\\
v_0 \in C_c^\infty(\Sigma).
\end{cases}
\end{equation}
We will construct a suitable solution $v_0\in C^\infty(\R^{n+1})$ to \eqref{eq:defv_0} explicitly in Appendix \ref{app:B}. The measurement function is defined as the restriction \begin{equation}\label{eq:measurementfun}
\psi:=v_0|_{\Sigma} \in C_c^\infty(\Sigma)
\end{equation} to the lateral boundary $\Sigma$.
The measurement function will be used in an integration by parts argument to cut off any contribution coming to the integral from the top $\Omega\times\{t=T\}$ of the time-cylinder $\Omega\times[0,T]$.
We denote
$$
\widetilde{\Sigma}=\Sigma^T\cap \supp(v_0).
$$

Our main result is that reconstruction of $a(x,t)$ from the non-linear map $\lambda_\psi$ is H\"older stable.

\begin{theorem}[Stability estimate with one dimensional measurements]\label{thm:stability}
Let $\Omega\subset \R^n$ be a bounded domain with a smooth boundary. Let $m\geq 2$ be an integer, $r\in \R$ with $r\leq s\in \N$ and $s+1>(n+1)/2$. Assume that $a_1,\s a_2\in C_c^{s+1}(\Omega\times \R)$ are admissible and $\psi$ is as in \eqref{eq:measurementfun}. Let $\Lambda_1,\s\Lambda_2 :H^{s+1}(\Sigma^T)\to H^r(\widetilde{\Sigma})$ be the Dirichlet-to-Neumann maps of the non-linear wave equation~\eqref{eq:intro_wave-eq}.

 Let $\eps_0>0$, $M>0$, $0<T<\infty$ and $\delta\in (0,M)$ be such that
 
\begin{equation}\label{est:DNmap_delta}
 \abs{\langle \psi, \Lambda_1(f)-\Lambda_2 (f)\rangle_{L^2(\widetilde{\Sigma})}} \leq \delta
\end{equation}
%
%
%
\noindent for all $f\in H^{s+1}(\Sigma^T)$ with $\Vert f\Vert_{H^{s+1}(\Sigma^T)}\leq \eps_0$.  Then 
\begin{equation}\label{eq:stability1D}
  \norm{a_1-a_2}_{L^\infty(\Omega \times [0,T])}\leq C \delta^{\sigma(s)},
\end{equation}
 where
 \[
 \sigma(s)= \begin{cases}
 {\frac{m-1}{(2m-1)(m(s+2)+1) }},&n=1,\\
 {\frac{m-1}{2n(2m-1)(m(s+2)+1) }}, & n\geq 2.
 \end{cases}
 \]
\end{theorem}
 Theorem~\ref{thm:uniqueness} follows from Theorem~\ref{thm:stability} by letting $\delta \to 0$.
Note that in the theorem we assume that our boundary values may be supported on all of $\Sigma^T$.
However, we only assume that the measurements are made on a smaller subset $\widetilde{\Sigma}=\Sigma^T\cap \supp(v_0)$ of $\Sigma^T$.

In fact, we emphasize that to recover the potential $a\in C^{s+1}_0(\Omega\times[t_1,t_2])$ in a stable way it is sufficient to make \emph{one dimensional measurements} 
\[
\lambda_{\psi}: f\mapsto\langle \psi, \Lambda(f)\rangle_{L^2(\widetilde{\Sigma})}\in\R\]
on $\widetilde{\Sigma}$.
Here $\psi$ can be considered as an instrument function that models the measurement instrument that is used to do observations on the solution $u$. Note that $\psi$ is a smooth function that is a constant on
$\partial\Omega\times[t_1,t_2]$.
This means that $a(x,t)$ can be recovered from \emph{low resolution measurements} if we can accurately control the source $f$.

\begin{corollary}\label{cor:Th1_1}
Let us adopt assumptions and notations of Theorem~\ref{thm:stability}. Instead of condition \eqref{est:DNmap_delta}, suppose that
$$
 \norm{\Lambda_1(f)-\Lambda_2(f)}_{H^r(\widetilde{\Sigma})}\leq \delta
$$
\noindent for all $f\in H^{s+1}(\Sigma^T)$ with $\Vert f\Vert_{H^{s+1}(\Sigma^T)}\leq \eps_0$. Then the stability estimate \eqref{eq:stability1D} is valid.
\end{corollary}

%
%

%

There are stability estimates for the recovery of the potentials $a$ and $b$ of the corresponding linear wave equation $\square u + b\p_t u + au=0$, see for example \cite{IsSun}, where the authors obtained a local H\"older stability result for this problem when given measurements on a part $\Gamma_0$ of the lateral boundary $\Sigma$.
In a related spirit, one might ask is it possible to recover a Riemannian metric $g$ when given the Dirichlet to Neumann map for the equation $(\partial_t^2-\Delta_g)u=0$.
Some earlier results in this direction are based on Tataru's unique continuation principle.
In this case, stability estimates are of logarithmic type, see e.g.~\cite{BoKuLa17}.
However, later these results have been improved by using different techniques.
In \cite{StUh} it was shown that a simple Riemannian metric $g$ can be recovered in a H\"older stable way from the DN map.



We also consider the question of reconstruction of the unknown potential when there is possibly noise $\mathcal{E}$ involved in the measurements. 
We assume that the noise is a bounded, possibly non-linear, mapping $H^{s+1}(\Sigma) \to H^r(\Sigma)$, $r\in \R$ and $r \leq s$.   We present our reconstruction and stability results with noise in $\R^{1+1}$.
The general case of $\R^{n+1}$, $n\geq 2$ will be given in Proposition~\ref{thm:noise_Rn}, see Section~\ref{sec:main-results-nD}. The reason is that the statement for $n\geq 2$ involves a Radon transformation. We write $\eps=0$ for the condition $\eps_1=\cdots=\eps_m=0$.

\begin{theorem}[Reconstruction, $n=1$]\label{thm:noise}
Let $\Omega\subset \R$ be an interval. Let $m\geq 2$ be an integer, $r\in \R$ with $r\leq s\in \N$ and $s+1>(1+1)/2$. Assume that $a \in C_c^\infty(\Omega\times \R)$ is admissible and let $\Lambda:H^{s+1}(\Sigma^T)\to H^r(\widetilde{\Sigma})$
be the Dirichlet-to-Neumann map of the non-linear wave equation~\eqref{eq:intro_wave-eq}. 
Assume also that $\E:H^{s+1}(\Sigma^T)\to H^r(\Sigma^T)$. 

Let $\eps_0>0$, $M>0$, $0<T<\infty$ and $\delta\in (0,M)$ be such that
 \[
  \norm{\E (f)}_{H^r(\Sigma^T)}\leq \delta, 
 \]
for all $f\in H^{s+1}(\Sigma^T)$ with $\Vert f\Vert_{H^{s+1}(\Sigma^T)}\leq \eps_0$. 
%

There exists $\tau\geq 1$, $\eps_1,\ldots,\eps_m>0$ and a finite family of functions $\{H_{j}^{\tau,(x_0,t_0)}\}\subset H^{s+1}(\Sigma^T)$
  where $j=1,\ldots,m$, $(x_0,t_0)\in \R^{1+1}$ such that
\begin{equation}\label{est:2D_noise}
\begin{aligned}
&\sup_{(x_0,t_0)\in \Omega\times [0, T]}\,\Big|
- a(x_0, t_0) \\
&\quad \quad \quad \quad \quad    -\frac{1}{2\pi}D_{\eps_1\cdots\eps_m}^m\big|_{\eps=0}\int_{\widetilde{\Sigma}} \psi\,(\Lambda + \E)(\eps_1H_1^{\tau,(x_0,t_0)} +\cdots +\eps_mH_m^{\tau,(x_0,t_0)}) dS 
\Big| \\
& \qquad  \leq C \delta^{\sigma (s)}.
\end{aligned}
\end{equation}
Here $\sigma(s)$ and $C$ are as in Theorem \ref{thm:stability} and the measurement function $\psi$ is as in \eqref{eq:measurementfun}.
The finite difference operator $D_{\eps_1\cdots\eps_m}^m\big|_{\eps=0}$ is defined in \eqref{eq:fin_diff}.

\end{theorem}

In higher dimensions the situation is somewhat different. Using a similar approach as in $1+1$ dimensions, we get an estimate similar to \eqref{est:2D_noise} for the Radon transform $\mathscr{R}(a)$ in place of $a$, see Proposition \ref{thm:noise_Rn} in Section \ref{sec:main-results-nD}. The knowledge of  $\mathscr{R}(a)$ allows us to get information of the unknown potential in a negative Sobolev index (by using the Fourier slice theorem, see Section~\ref{RT}).
Then Theorem \ref{thm:stability} for higher dimensions $n\geq 2$ follows by combining this fact with a standard interpolation argument where we use the admissible property of the potential $a$. In fact, the term $1/(2n)$ in the exponent $\sigma(s)$ comes from this interpolation step. The definition of the Radon transform and its relevant properties can be found in Section \ref{sec:main-results-nD}.

Let us explain how we prove Theorem~\ref{thm:stability}. The proof is based on the higher order linearization method, which was used in many of the works mentioned earlier. We now explain this method. We will also use an integration by parts argument introduced in the study of partial data inverse problem for non-linear elliptic equations in \cite{LLLS19b, KrUh20}. Similar argument was also used recently in~\cite{HUZ20}.

%

We first explain how we can recover $a$ from the DN map $\Lambda$ of the equation~\eqref{eq:intro_wave-eq}.
 Let us consider the case $m=2$. Let $f_1,\s f_2\in H^{s+1}(\Sigma)$, and let us denote by $u_{\eps_1f_1+\eps_2f_2}$ the solution to~\eqref{eq:intro_wave-eq} with boundary data $\eps_1f_1+\eps_2f_2$, where $\eps_1,\eps_2$ are sufficiently small parameters. By taking the mixed derivative of the equation~\eqref{eq:intro_wave-eq} with respect to the parameters $\eps_1$ and $\eps_2$, and of the solution $u_{\eps_1f_1+\eps_2f_2}$, we see that 
\[
 w:=\frac{\p}{\p \eps_1}\frac{\p}{\p \eps_2}\Big|_{\eps_1=\eps_2=0} u_{\eps_1 f_1 + \eps_2 f_2}
\]
solves 
\begin{equation}\label{eq:second_deriv}
 \square w   = -2a v_{1}v_2
\end{equation}
 with zero initial and boundary data. 
 Here the functions $v_j$ solve
 \begin{equation}
 \begin{cases}
 \square v_j = 0, &\text{in } \Omega\times [0,T],\\ 
 v_j = f_j, &\text{on } \p\Omega\times[0,T],\\
 v_j\big|_{t=0} = \p v_j\big|_{t=0} = 0, &\text{in } \Omega,
 \end{cases}
 \end{equation}
 for $j=1,2$.  
  This way we have produced new linear equations from the non-linear equation \eqref{eq:intro_wave-eq}. Studying these new equations in inverse problems for non-linear equations is known as the higher order linearization method.

 If we assume that the DN map is known, then the normal derivative of $w$ is also known on $\Sigma$ since  
  \[
   \p_\nu w = \p^2_{\eps_1\eps_2}|_{\eps_1=\eps_2=0}\s \Lambda(\eps_1f_1+\eps_2f_2).
 \]
We let $v_0$ be an auxiliary function solving $\square v=0$ with $v_0|_{t=T} =\p_t v_0|_{t=T} = 0$ in  $\Omega$. The function $v_0$ will compensate the fact we know $\p_\nu w$ only on the lateral boundary $\Sigma$. Then, by multiplying~\eqref{eq:second_deriv} by $v_0$, and integrating by parts on $\Omega\times [0,T]$, we have the integral identity
 \begin{equation}\label{eq:integral_identity_derivitve}
  \int_{\Sigma}v_0\s \p^2_{\eps_1\eps_2}|_{\eps_1=\eps_2=0}\Lambda(\eps_1f_1+\eps_2f_2)\s \d S=\int_{\Omega\times [0,T]} \s v_0\s \square w \s \d x \s \d t  =-2\int_{\Omega\times [0,T]}a\s v_0\s v_1\s v_2\s \d x \s \d t. 
 \end{equation}
 Thus the quantity
 \begin{equation}\label{eq:integral_denisty}
  \int_{\Omega\times [0,T]} a\s v_0\s v_1\s v_2 \s \d x \s \d t 
 \end{equation}
 is known from the knowledge of the DN map. Since $v_1$ and $v_2$ were arbitrary solutions to $\square v=0$, we may choose suitable solutions $v_1$ and $v_2$ so that the products of the form $v_0v_1v_2$ become dense in $L^1(\Omega\times [0,T])$. This recovers $a$. 
 
Heuristically, in $1+1$ dimensions it would be sufficient to have $v_1=\delta((x-x_0)-(t-t_0))$ and $v_2=\delta((x-x_0)+(t-t_0))$ to recover $a(x_0,t_0)$ for $(x_0,t_0)\in \R^{1+1}$. Here $\delta$ is the $1$-dimensional delta function. In this case $v_1v_2$ is the delta function $\delta_{(x_0,t_0)}$ of $\R^{1+1}$ with mass at $(x_0,t_0)$. 
However, since delta distributions do not correspond to a family of solutions of the non-linear wave equation, we will instead use \emph{approximate delta functions}.
  In higher dimensions, different choices of $v_1$ and $v_2$ reduce the integral~\eqref{eq:integral_denisty} to a Radon transformation of $a$ on $\R^n$, which is stably invertible.

Instead of differentiating equation~\eqref{eq:intro_wave-eq}, to obtain stability we will take the mixed finite difference
  $D_{\eps_1,\eps_2}^2$ of $u_{\eps_1f_1+\eps_2f_2}$. In this case, we have the following integral identity 
\begin{equation}
\begin{aligned}
  \int_{\Omega\times [0,T]} a\s v_0\s v_1\s v_2\s \d x \s \d t &=\int_{\Sigma}v_0\s D_{\eps_1\eps_2}^2\Big|_{\eps_1=\eps_2=0}\Lambda(\eps_1f_1+\eps_2f_2) \s \d S\\
  &\qquad  + \frac{1}{\eps_1\eps_2}\int_{\Omega\times [0,T]} v_0\square\s \tildeR\s \d x \s \d t ,
  \end{aligned}
 \end{equation}
where $\tildeR=\mathcal{O}_{E^{s+2}}(\langle\eps_1,\eps_2\rangle^3)$ in an energy space norm, for details see \eqref{eq:energy_norm} and \eqref{eq:tildeR}--\eqref{est:square_tildeR}. Here we denote by $\langle\eps_1,\eps_2\rangle^3$ homogeneous polynomials of order $3$ in $\eps_1$ and $\eps_2$. Stability result Theorem \ref{thm:stability} will follow by optimizing in $\eps_1$ and $\eps_2$ and parameters related to the solutions $v_1$ and $v_2$. The proof of Theorem~\ref{thm:noise} follows from a similar argument.

 The paper is organized as follows. In Section \ref{sec:forward-problem}, we lay out the basic properties for semi-linear hyperbolic equations that we use. This includes the well-posedness of the boundary value problem for the equation~\eqref{eq:intro_wave-eq}. 
	We also calculate formulas for the second order finite differences of solutions to~\eqref{eq:intro_wave-eq} in Section~\ref{sec:forward-problem}. In Section~\ref{sec:main-results-1D}, we prove Theorems~\ref{thm:stability} and~\ref{thm:noise} in $1+1$ dimensions, and in Section~\ref{sec:main-results-nD} we prove these theorems in higher dimensions. We have placed some proofs in Appendix. 

\section{Forward problem and definition of the DN map}\label{sec:forward-problem}
In this section we study the existence of solutions to the boundary value problem of a non-linear wave equation in $\R^{n+1}$:
\begin{equation*}
\begin{cases}
\square u + au^m   =0, & \text{in } \Omega\times[0,T],\\
u= f, & \text{on } \p\Omega\times[0,T], \\
u\big|_{t=0} =\partial_t u\big|_{t=0} = 0,&\text{in }\Omega.
\end{cases}
\end{equation*}
Let $\Omega$ be an open subset of $\R^{n}$ with smooth boundary. Let $s\in \N$ and let us denote for the sake of brevity
\[
X^{s}(\Omega):= C([0,T]\; H^{s}(\Omega))\cap C^{s}([0,T]\; L^2(\Omega)).
\]
If there is no danger of misunderstanding, we simply denote $X^s(\Omega)$ by $X^s$, or just by $X$ if the index $s$ is additionally known from the context. The norm of the Banach space $X^s(\Omega)$ is given by
\[
\Vert f \Vert_{X^s} := \sup_{0<t<T} \Big(
\Vert f(\ccdot,t)\Vert_{H^{s+1}(\Omega)} + \Vert \partial_t f(\ccdot,t)\Vert_{H^s(\Omega)}
\Big).
\]

To prove existence of small solutions for the non-linear wave equation, we consider the linear initial-boundary value problem
\begin{equation*}
\begin{cases}
\square u  = F, \quad \text{in } \Omega\times[0,T],\\
u=f, \,\quad\ \text{ on } \p\Omega\times[0,T] ,\\
u\big|_{t=0} = \psi_0,\quad \partial_t u\big|_{t=0} = \psi_1,&\text{in }\Omega
\end{cases}
\end{equation*}
 for the linear wave operator.
The standard \emph{compatibility conditions of order $s$} for this problem are given as 
\begin{equation}\label{eq:compatibility}
\begin{aligned}
 f|_{t=0}&=\psi_0|_{\p\Omega}, \quad \p_tf|_{t=0}=\p_tu|_{\p\Omega \times \{0\}}=\psi_1|_{\p \Omega}, \\
 \p_t^2f|_{t=0}&=\p_t^2u|_{\p \Omega \times \{0\}}=\Delta \psi_0|_{\p \Omega} +F|_{\p\Omega\times \{0\}}, 
 \end{aligned}
\end{equation}
and similarly for the higher order derivatives up to order $s$. These conditions guarantee that at the boundary $\p\Omega$ the initial data $(\psi_0,\psi_1)$ matches with the corresponding boundary condition $f$, see \cite[Section 2.3.7]{KKL01}.
Especially, if $\p_t^kf|_{t=0}=0$ for all $k=0,\ldots,s$, and $F\equiv 0$ and $\psi_0\equiv\psi_1\equiv 0$, then the compatibility conditions of order $s$ are true.
We will use the following result from the book~\cite[Theorem 2.45]{KKL01}, see also~\cite{LLT86}.
\begin{theorem*}[Existence and estimates for linear equation \cite{KKL01}]\label{thm:KKL-energy}
Let $s\in\N$ and $0<T<\infty$. Assume that $F\in L^1([0,T]\; H^s(\Omega))$, $\p_t^sF\in L^1([0,T]\; L^2(\Omega))$, $\psi_0\in H^{s+1}(\Omega)$, $\psi_1\in H^s(\Omega)$ and $f\in H^{s+1}(\Sigma)$. If all the compatibility conditions up to the order $s$ are satisfied, then the problem 
\begin{equation}\label{wave-eq_with_data}
\begin{cases}
\square u  = F, \quad\text{in } \Omega\times[0,T],\\
u=f, \quad\text{on } \p\Omega\times[0,T], \\
u\big|_{t=0} = \psi_0,\quad \partial_t u\big|_{t=0} = \psi_1,&\text{in }\Omega
\end{cases}
\end{equation}
has a unique solution $u$ satisfying
\[
u\in X^{s+1}(\Omega) \text{ and } \p_\nu u|_{\Sigma}\in H^{s}(\Sigma). 
\]
Moreover, we have the estimate for all $t\in [0, T]$
\begin{equation}\label{energy_estim}
\begin{split}
 &\norm{u(\ccdot,t)}_{H^{s+1}(\Omega)}+\norm{\p_t^{s+1}u(\ccdot,t)}_{L^2(\Omega)}+\norm{\p_\nu u}_{H^{s}(\Sigma)} \\
 &\qquad\leq c_T\left(\norm{F}_{L^1([0,T]\; H^s(\Omega))}+\norm{\p_t^s F}_{L^1([0,T]\; L^2(\Omega))}\right.  \\
 &\qquad\qquad\qquad\qquad\qquad+\left.\norm{\psi_0}_{H^{s+1}(\Omega)}+\norm{\psi_1}_{H^{s}(\Omega)}+\norm{f}_{H^{s+1}(\Sigma)}\right).
 \end{split}
\end{equation}
\end{theorem*}

Let us define the \emph{energy spaces} $E^s$ (see e.g.~\cite[Definition 3.5 on page 596]{CB08}) of functions in $\Omega\times [0,T]\subset \R^{n+1}$:
\[
 E^s=\bigcap_{0\leq k \leq s}C^k([0,T]\; H^{s-k}(\Omega)).
\]
These spaces are equipped with the norm
\begin{equation}\label{eq:energy_norm}
 \norm{u}_{E^s}=\sup_{0<t<T}\sum_{0\leq k \leq s}\norm{\p_t^ku(\ccdot,t)}_{H^{s-k}(\Omega)}.
\end{equation}
The reason why we are considering the spaces $E^s$ is that if $s>(n+1)/2$, then $E^s$ is an algebra (see e.g.~\cite{CB08}) and we have the norm estimate
\begin{equation}\label{banach_algebra}
 \norm{u\s v}_{E^s}\leq C_s \norm{u}_{E^s}\norm{v}_{E^s}, \text{ for all } u,v\in E^s.
\end{equation}

We record the following consequence of Theorem~\ref{thm:KKL-energy}, which we will use extensively. We have placed its proof in the Appendix~\ref{app:A}.
\begin{corollary}\label{cor:enery-estimate}
 Adopt the notation and assumptions of Theorem~\ref{thm:KKL-energy}. Assume in addition that
 \[
  \p_t^kF\in L^1([0,T]\; H^{s-k}(\Omega)), \quad k=0,1,\ldots,s.
 \]
Then the solution $u$ to~\eqref{wave-eq_with_data} satisfies 
\[
u\in E^{s+1}(\Omega) \text{ and } \p_\nu u|_{\Sigma}\in H^{s}(\Sigma)
\]
and
 \begin{equation}\label{energy_estim_Es}
\begin{split}
 &\norm{u}_{E^{s+1}} +\norm{\p_\nu u}_{H^{s}(\Sigma)} \leq c_{s,T}\Big( \sum_{0\leq k \leq s}\norm{\p_t^kF}_{L^1([0,T]\; H^{s-k}(\Omega))}  \\
 &\qquad\qquad\qquad\qquad\qquad\qquad\qquad+\norm{\psi_0}_{H^{s+1}(\Omega)}+\norm{\psi_1}_{H^{s}(\Omega)}+\norm{f}_{H^{s+1}(\Sigma)}\Big).
 \end{split}
\end{equation}
\end{corollary}

The proofs of the following results are quite standard, we postpone them to the Appendix~\ref{app:A} for the interested reader.

\begin{lemma}\label{lemma:nonlinear-solutions}
Let $s+1>(n+1)/2$. Suppose  that $a\in C_c^\infty(\Omega\times [0,T])$ satisfies $\norm{a}_{C^s(\Omega\times[0,T])}\leq L$.
There is $\kappa>0$ and $\rho>0$ such that if $f\in H^{s+1}(\Sigma)$ satisfies $\Vert f\Vert_{H^{s+1}} \leq \kappa$ and $\p_t^kf|_{t=0}=0$, $k=0,\ldots,s$, on $\p\Omega$,
there is a unique solution to 
\begin{equation}\label{eq:nonlinear_equation_lemma}
\begin{cases}
\square u + au^m   =0, &\text{in } \Omega\times[0,T],\\
u=f, &\text{on } \p\Omega\times[0,T], \\
u\big|_{t=0} = \partial_t u\big|_{t=0} = 0,&\text{in }\Omega
\end{cases}
\end{equation}
in the ball 
\[
B_\rho(0):= \{u \in E^{s+1} \mid \Vert u \Vert_{E^{s+1}} < \rho\}\subset E^{s+1}.
\]
Furthermore, the solution satisfies the estimate
\begin{equation}\label{eq:solution-estimate1}
\Vert u\Vert_{E^{s+1}} \leq C_0\s  \Vert f\Vert_{H^{s+1}(\Sigma)},
\end{equation}
where $C_0>0$ is a constant depending only on $s$ and $T$.
\end{lemma}
%
%

We are ready to consider an inverse problem for the non-linear hyperbolic equation
\begin{equation}\label{eq:non-linear-wave}
\begin{cases}
\square u + au^m   =0, &\text{in } \Omega\times[0,T],\\
u=f, &\text{on } \p\Omega\times[0,T], \\
u\big|_{t=0} = \partial_t u\big|_{t=0} = 0,&\text{in }\Omega.
\end{cases}
\end{equation}
Our measurement data is the Dirichlet-to-Neumann map, which is a map from a small ball in $H^{s+1}(\Sigma)$ into $H^{s}(\Sigma)$ and is defined as follows.
\begin{definition}[Dirichlet-to-Neumann map]
 Let $\Omega$ be an open subset of $\R^{n}$ and let $s+1>(n+1)/2$, $s\in \N$. 
 Let $\rho>0$ be such for all $f$ with $\norm{f}_{H^{s+1}(\Sigma)}<\kappa$ and $\p_t^kf|_{t=0}=0$, $k=0,\ldots, s$, the problem~\eqref{eq:intro_wave-eq} has a unique solution $u\in E^{s+1}$ satisfying $\norm{u}_{E^{s+1}}<\rho$. The Dirichlet-to-Neumann map $\Lambda$ is the map $\{f\in H^{s+1}(\Sigma): \norm{f}_{H^{s+1}(\Sigma)}<\kappa\} \to H^s(\Sigma)$ given as
 \begin{equation}\label{C2N-map}
  \Lambda(f)=\p_\nu u \text{ on } \Sigma, \quad f\in H^{s+1}(\Sigma), \ \norm{f}_{H^{s+1}(\Sigma)}<\kappa, 
 \end{equation}
 where $u$ is the unique solution to~\eqref{eq:intro_wave-eq} with $\norm{u}_{E^{s+1}}<\rho$.
\end{definition}

We end this section by an expansion formula for a family of solutions depending on small parameters.

\begin{proposition}\label{cor:epsilon_expansion}
Let $s+1> (n+1)/2$. 
Suppose that $a\in C_c^\infty([0,T]\times \Omega)$ and $\norm{a}_{C^{s+1}}\leq L$. There is $\kappa>0$ and $\rho>0$ with the following property: If $f_j\in H^{s+1}(\Sigma)$ and $\eps_j>0$ satisfy $\Vert \eps_1f_1+\cdots\eps_mf_m\Vert_{H^{s+1}} \leq \kappa$ and $\p_t^kf_j|_{t=0}=0$ on $\p\Omega$, $k=0,\ldots,s$, $j=1,\ldots, m$,
 then there exists 
a unique solution $u$ to
 \begin{equation}\label{eq:epsilons}
\begin{cases}
\square u + au^m   =0, \qquad \qquad\qquad \qquad \, \text{in } \Omega\times[0,T],\\
u=\eps_1f_1+\eps_2f_2+\cdots +\eps_mf_m, \quad \, \text{ on } \p\Omega\times[0,T], \\
u\big|_{t=0} = 0,\quad \partial_t u\big|_{t=0} = 0,\qquad\qquad \text{in } \Omega
\end{cases}
\end{equation}
in the ball 
\[
B_\rho(0):= \{u \in E^{s+1} \mid \Vert u \Vert_{E^{s+1}} < \rho\}\subset E^{s+1}.
\]
The solution satisfies the estimate
\begin{equation}\label{eq:solution-estimate2}
\Vert u\Vert_{E^{s+1}} \leq C_0\s  \Vert \varepsilon_1 f_1 +\cdots + \varepsilon_m f_m\Vert_{H^{s+1}(\Sigma)},
\end{equation}
where $C_0>0$ is a constant depending only on $s$, $T$ and $L$. Furthermore, $u$ has the following expansion in $\eps_1,\ldots, \eps_m$ in terms of the multinomial coefficients
\begin{equation}\label{id:expam_epsilons}
u=\epsilon_1 v_1+\cdots + \epsilon_m v_m +\sum_{k_1,\s k_2,\ldots, k_m}\binom{m}{k_1,\s k_2,\cdots, k_m}\epsilon_1^{k_1}\cdots \epsilon_m^{k_m}w_{(k_1,\ldots, k_m)} 
+\mathcal{R}.
\end{equation}
Here for $j=1,\ldots, m$ the functions $v_j$ satisfy
\begin{equation}\label{eq:norm_21}
\begin{cases}
\square v_j=0, \quad \text{in } \Omega\times[0,T],\\
v_j=f_j, \quad \text{ on } \p\Omega\times[0,T], \\
v_j\big|_{t=0} = 0,\quad \partial_t v_j\big|_{t=0} = 0,\quad \text{in } \Omega
\end{cases}
\end{equation}
and for $k_j\in \{1,\ldots, m\}$ the functions $w_{k_1,\ldots, k_m}$ satisfy
\begin{equation}\label{eq:norm_2}
\begin{cases}
\square w_{k_1,\ldots, k_m} +  a\s  v_1^{k_1}\cdots v_m^{k_m}  = 0, \quad \text{in } \Omega\times[0,T],\\
w_{k_1,\ldots, k_m}=0, \quad \text{ on } \p\Omega\times[0,T], \\
w_{k_1,\ldots, k_m}\big|_{t=0} = 0,\quad \partial_t w_{k_1,\ldots, k_m}\big|_{t=0} = 0,\quad \text{in } \Omega
\end{cases}
\end{equation}
and
\begin{equation}\label{est:square_R}
\begin{split}
\norm{\mathcal{R}}_{E^{s+2}}&\leq c(s, T)\s   \norm{a}_{E^{s+1}}^2 \Vert \varepsilon_1 f_1 +\cdots + \varepsilon_m f_m\Vert_{H^{s+1}(\Sigma)}^{2m-1},  \\
\norm{\square\, \mathcal{R}}_{E^{s+1}} &\leq C(s, T)\s   \norm{a}_{E^{s+1}}^2 \Vert \varepsilon_1 f_1 +\cdots + \varepsilon_m f_m\Vert_{H^{s+1}(\Sigma)}^{2m-1}.
\end{split}
\end{equation}
\end{proposition}
\begin{proof}
First, equation \eqref{eq:solution-estimate2} follows from \eqref{eq:solution-estimate1}.
Then we note that $\mathcal{F}= u- (\eps_1 v_1+\eps_2v_2+\cdots + \eps_m v_m)$ satisfies
 \[
\begin{cases}
\square \mathcal{F} =-au^m, & \text{in } \Omega\times[0,T],\\
\mathcal{F}=0, & \text{on } \p\Omega\times[0,T], \\
\mathcal{F}\big|_{t=0} = 0,\partial_t \mathcal{F}\big|_{t=0} = 0,&\text{in } \Omega.
\end{cases}
\]
Hence, by~\eqref{eq:solution-estimate2} and by using the energy estimate from Corollary \ref{cor:enery-estimate}, one obtains
\begin{equation}\label{est:iterative}
\norm{\mathcal{F}}_{E^{s+2}} \leq C(s, T) \norm{au^m}_{E^{s+1}} \leq C(s, T) \s \norm{a}_{E^{s+1}}  \s  \Vert \varepsilon_1 f_1 +\cdots + \eps_m f_m\Vert_{H^{s+1}(\Sigma)}^m.
\end{equation}
Here we have used that $E^{s+1}$ is an algebra and the estimate \eqref{eq:solution-estimate2}:
\[
 \norm{u}_{E^{s+1}}\leq C \norm{\varepsilon_1 f_1 +\cdots + \eps_m f_m}_{H^{s+1}(\Sigma)}.
\]

One step further, taking into account \eqref{eq:norm_2}, the function $\mathcal{R}$ given by
\begin{equation}\label{remainder_term_u}
 \mathcal{R}:= u-(\epsilon_1 v_1+ \cdots + \epsilon_m v_m)- \sum_{k_1,\s k_2,\ldots, k_m}\binom{m}{k_1,\s k_2,\cdots, k_m}\epsilon_1^{k_1}\cdots \epsilon_m^{k_m}w_{k_1,\ldots, k_m} 
 \end{equation}
 satisfies 
 \begin{equation}\label{equation:remainder_term}
\begin{cases}
\square \mathcal{R} =-au^m + a(\eps_1v_1+ \eps_2 v_2+\cdots + \eps_m v_m)^m, & \text{in } \Omega\times[0,T],\\
\mathcal{R}=0, &\text{on } \p\Omega\times[0,T], \\
\mathcal{R}\big|_{t=0} = 0,\quad \partial_t \mathcal{R}\big|_{t=0} = 0,&\text{in } \Omega.
\end{cases}
\end{equation}
Using this identity together with the estimate \eqref{eq:solution-estimate2} and \eqref{est:iterative}, we obtain
\begin{align*}
\norm{\square&\, \mathcal{R}}_{E^{s+1}} \leq C(s, T)\norm{-au^m + a(\eps_1v_1+\cdots +\eps_m v_m)^m}_{E^{s+1}}\\
& = C(s, T)\norm{a\s (u-(\eps_1v_1+\cdots \eps_m v_m)) \s P_{m-1}(u,\eps_1v_1+\cdots +\eps_m v_m)}_{E^{s+1}}\\
& \leq C(s,T) \norm{a}_{E^{s+1}}\norm{\mathcal{F}}_{E^{s+2}}\norm{P_{m-1}(u,\eps_1v_1+\cdots +\eps_m v_m)}_{E^{s+1}} \\ 
& \leq C(s,T) \norm{a}_{E^{s+1}}^2 \s  \Vert \varepsilon_1 f_1 +\cdots +\varepsilon_m f_m\Vert_{H^{s+1}(\Sigma)}^m \\
& \qquad \qquad  \qquad \qquad  \times \Big(\sum_{l=0}^{m-1}\norm{u^{m-1-l}\s (\eps_1v_1+\cdots +  \eps_m v_m)^l }_{E^{s+1}}\Big), \\
& \leq C(s,T) \norm{a}_{E^{s+1}}^2 \s  \Vert \varepsilon_1 f_1 +\cdots + \varepsilon_m f_m\Vert^{2m-1}_{H^{s+1}(\Sigma)}.
\end{align*}
Here we wrote
   \begin{equation*}
     u^m-v^m=(u-v)P_{m-1}(u,v),
    \end{equation*}
where $P_{m-1}(a,b)=\sum_{k=0}^{m-1} a^{m-1-k}b^k$. In the last inequality we used~\eqref{eq:solution-estimate2}. Thus it follows from the energy estimate~\eqref{energy_estim_Es} that
\[
 \norm{\mathcal{R}}_{E^{s+2}}\leq c(s, T)\s   \norm{a}_{E^{s+1}}^2 \Vert \varepsilon_1 f_1 +\cdots + \varepsilon_m f_m\Vert_{H^{s+1}(\Sigma)}^{2m-1}.
\qedhere\]
\end{proof}

We will calculate mixed finite differences $D^m_{\eps_1\eps_2\cdots\eps_m}$ of the solution $u$ of~\eqref{eq:epsilons}. We have that
\begin{equation}\label{eq:mixed_difference}
 D^m_{\eps_1\eps_2\cdots\eps_m}\big|_{\eps=0} u=m!\s w_{1,1,\ldots,1}+ D^m_{\eps_1\eps_2\cdots\eps_m}\big|_{\eps=0}\mathcal{R}.
\end{equation}
Here we write $\eps=0$ when $\eps_1=\ldots=\eps_m=0$.
For more details, we refer the reader to Appendix \ref{app:C}.
The finite difference is defined as usual by
\begin{equation}\label{eq:fin_diff}
D_{\eps_1,\ldots,\eps_m}^m\big|_{\eps=0} u_{\eps_1f_1+\cdots+\eps_mf_m}
=
\frac{1}{\eps_1\cdots\eps_m}\sum_{\sigma\in\{0,1\}^m}
(-1)^{|\sigma|+m}u_{\sigma_1\eps_1 f_1+\ldots+\sigma_m\eps_m f_m}.
\end{equation}
For example, when $m=2$, we have
\[
D^2_{\epsilon_1, \epsilon_2}\big|_{\eps_1=\eps_2=0} u:= \frac{1}{\epsilon_1 \epsilon_2} \left(u_{\eps_1f_1+\eps_2f_2}- u_{\eps_1f_1}-u_{\eps_2f_2}\right).
\]
Here we used the fact that the solution to \eqref{eq:epsilons} with $\eps_1=\eps_2=0$ is identically zero.

Let $v_0$ be an auxiliary function solving $\square v_0=0$ with $v_0|_{t=T} =\p_t v_0|_{t=T} = 0$ in  $\Omega$. By integrating by parts we obtain
 \begin{align}
 &\int_{\Sigma}v_0\s D_{\eps_1,\ldots,\eps_m}^m\big|_{\eps=0}\Lambda(\eps_1f_1+\cdots+\eps_mf_m) \s \d S \\
 &\quad=\int_{\Sigma}v_0\s D_{\eps_1,\ldots,\eps_m}^m\big|_{\eps=0}\p_\nu u_{\eps_1f_1+\cdots+\eps_mf_m} \s \d S \\
 &\quad=m!\int_{\Omega\times[0,T]} v_0 \square\s w_{1,1,\ldots,1} \s \d x \s \d t 
 +
  \frac{1}{\eps_1\cdots\eps_m}\int_{\Omega\times[0,T]} v_0 \square\s\widetilde{\mathcal{R}}\s \d x \s \d t . 
\end{align}
Here we denoted 
\begin{equation}\label{eq:tildeR}
\tildeR := \eps_1\eps_2 \ldots \eps_m D_{\eps_1,\ldots,\eps_m}^m\big|_{\eps=0} \mathcal{R},
\end{equation} 
and $\tildeR$ satisfies
\begin{equation}\label{est:square_tildeR}
\begin{split}
\norm{\widetilde{\mathcal{R}}}_{E^{s+2}}&\leq c(s, T)\s   \norm{a}_{E^{s+1}}^2 \sum_{\sigma\in\{0,1\}^m}
(-1)^{|\sigma|+m}\Vert \sigma_1\varepsilon_1 f_1 +\cdots + \sigma_m\varepsilon_m f_m\Vert_{H^{s+1}(\Sigma)}^{2m-1},  \\
\norm{\square\, \widetilde{\mathcal{R}}}_{E^{s+1}} &\leq C(s, T)\s   \norm{a}_{E^{s+1}}^2 \sum_{\sigma\in\{0,1\}^m}
(-1)^{|\sigma|+m}\Vert \sigma_1\varepsilon_1 f_1 +\cdots + \sigma_m\varepsilon_m f_m\Vert_{H^{s+1}(\Sigma)}^{2m-1}.
\end{split}
\end{equation}
We have arrived to the following integral identity which connects the potential $a$ with the DN-map $\Lambda$.\\
\noindent\textbf{Integral identity:}
\begin{equation}\label{eq:integral_identity_finite_difference}
\begin{aligned}
  -m!\int_{\Omega\times [0,T]} a\s v_0\s v_1\s v_2\cdots v_m \s \d x \s \d t 
  &=
\int_{\Sigma}v_0\s D_{\eps_1,\ldots,\eps_m}^m\big|_{\eps=0}\Lambda(\eps_1f_1+\cdots+\eps_mf_m)\s \d S\\&\qquad + \frac{1}{\eps_1\eps_2\cdots\eps_m}\int_{\Omega\times [0,T]} v_0\square\s \widetilde{\mathcal{R}}\s \d x \s \d t .
\end{aligned}
\end{equation}
We will use this identity several times throughout the text.
%
 
%


\begin{remark}
By taking $\eps_j\to 0$, the integral identity \eqref{eq:integral_identity_finite_difference} implies 
\begin{equation}\label{eq:integral_identity_finite_difference B}
\begin{aligned}
 \int_{\Omega\times [0,T]} a\s v_0\s v_1\s v_2\cdots v_m \s \d x \s \d t 
  &=-\frac 1{m!}
\int_{\Sigma}\psi \s \p_{\eps_1}\p_{\eps_2}\ldots \p_{\eps_m}\big|_{\eps=0}\Lambda(\eps_1f_1+\cdots+\eps_mf_m)\s \d S,
\end{aligned}
\end{equation}
where $\psi=v_0|_{\Sigma}$ is a measurement function and $v_j$, $j=0,1,2,\dots,m$, are the solutions of the linearized equation $\square v_j=0$. We note that similar identities 
are encountered in the study of inverse problems for elliptic equations, e.g. $\Delta U(x)+q(x)U
(x)^m=0,$ $U|_{\p \Omega}=f$, with the solutions $V_j$ of the linearized equation $\Delta V_j(x)=0$, see \cite{LLLS19b}.
As the constant function $V_0=1$ satisfies the linearized equation, one could use a similar approach to the one used in this paper to study the above non-linear elliptic equation with the one-dimensional
boundary map $f\mapsto \langle \Psi,\p_\nu U|_{\p \Omega}\rangle_{L^2(\p \Omega)}$ and the  
measurement function $\Psi=1$. However, these considerations are outside the context of this paper.
\end{remark}

\section{Proofs of the main results}\label{sec:main-results-1D}

\subsection{Proof of Theorems \ref{thm:stability} and \ref{thm:noise} in $1+1$ dimensions with $m=2$}
We prove Theorem \ref{thm:stability} in $\R^{1+1}$ with $m=2$ separately. We do this to present the main point of the proof of Theorem \ref{thm:stability} in a simple case. 
The proof will be divided into three steps. 

\noindent\textbf{Step 1.} Let $\eps_j>0$, $j=1,2$, and $f_j\in H^{s+1}(\Sigma)$ be functions that satisfy $\p_t^kf_j|_{t=0}=0$, $k=0,\ldots,s$, on $\p\Omega$. 
Suppose also that $\norm{\eps_1 f_1 +\eps_2f_2}_{H^{s+1}(\Omega\times[0,T])} \leq \kappa$ for $\kappa>0$ small enough, as in Lemma \ref{lemma:nonlinear-solutions}.
%
Then, for $l=1,2$, we have that the problem
 \begin{equation}
\begin{cases}
\square u_{l} + a_l\s u_l^2   =0,  \  \ \qquad \qquad \text{in } \Omega\times[0,T],\\
u_l=\eps_1f_1+\eps_2f_2, \ \, \quad \qquad \text{ on } \p\Omega\times[0,T],\\
u_l\big|_{t=0} = 0,\quad \s \partial_t u_l\big|_{t=0} = 0,\text{ in } \Omega
\end{cases}
\end{equation}
has a unique solution $u_l$ with an expansion of the form
\begin{equation}\label{eq:expansion_ul_proof}
 u_{l}=\eps_1v_{l,1}+\eps_2v_{l,2} + 2 \eps_1\eps_2 w_{l,(1,1)} +\eps_1^2w_{l,(2,0)}+\eps_2^2w_{l,(0,2)} +\mathcal{R}_l,
\end{equation}
where $v_{l,j}$ and $w_{l,(k_1,k_2)}$, $l,k_1,k_2=1,2$, solve~\eqref{eq:norm_21} and~\eqref{eq:norm_2} with $a$ replaced with $a_l$. Note that since the equation for $v_{l,j}$ is independent of $a_l$, we have by the uniqueness of solutions that
\[
 v_{1,j}=v_{2,j}=:v_j, \quad j=1,2.
\]

The correction term $\mathcal{R}_l$ satisfies
\begin{equation}\label{eq:estim_for_Rl}
 \norm{\square\, \mathcal{R}_l}_{E^{s+1}} \leq C(s, T)\s   \norm{a_l}_{E^{s+1}}^2 \Vert \varepsilon_1 f_1 + \varepsilon_2 f_2\Vert_{H^{s+1}(\Sigma)}^{3}.
\end{equation}
We have that the mixed second difference $ D^2_{\eps_1\eps_2}\big|_{\eps_1=\eps_2=0}$ of $u_l$ is 
\[
 D^2_{\eps_1\eps_2}\big|_{\eps_1=\eps_2=0}u_l=2 w_{l,(1,1)}+\frac{1}{\eps_1\eps_2}\mathcal{R}_l.
\]
Consequently
\begin{equation}\label{eq:difference_finite}
 \square\s D^2_{\eps_1\eps_2}\big|_{\eps_1=\eps_2=0}u_l=- 2 a_l v_{1}v_{2}+\frac{1}{\eps_1\eps_2}\square\tildeR_l,
\end{equation}
where $\tildeR_l := \eps_1\eps_2  D^2_{\eps_1\eps_2}\big|_{\eps_1=\eps_2=0}\mathcal{R}_l$ similarly as in \eqref{eq:tildeR}.

As the first step we derive a useful integral identity which relates the DN maps $\Lambda_1$ and $\Lambda_2$ 
with the information of the unknown potentials $a_1$ and $a_2$ in $\Omega\times [0, T]$. 
We recall that $v_0$ is an auxiliary function given in \eqref{eq:defv_0}.
%
%
Combining \eqref{eq:difference_finite} and the fact that $v_0$ satisfies $\square v_0=0$ with $v_0|_{t=T} =\p_t v_0|_{t=T} = 0$ in  $\Omega$ we get
\begin{equation}\label{id:integral_identity}
\begin{aligned}
&   \int_{\Sigma}v_0  \s D^2_{\eps_1\eps_2}\big|_{\eps_1=\eps_2=0}\left[(\Lambda_1- \Lambda_2)(\eps_1f_1+ \eps_2f_2)\right]\d S\\
&\qquad=  \int_{\Sigma} v_0 \s D^2_{\eps_1\eps_2}\big|_{\eps_1=\eps_2=0}\left[ \partial_\nu u_{1}-\partial_\nu u_{2} \right]\d S\\
&\qquad=  \int_{\partial \Omega \times [0, T]} v_0 \partial_\nu \left[\s D^2_{\eps_1\eps_2}\big|_{\eps_1=\eps_2=0} (u_{1}-u_{2} )\right]\d S\\
&\qquad =\int_{\Omega \times [0, T]} v_0 \left[ \square (\s D^2_{\eps_1\eps_2}\big|_{\eps_1=\eps_2=0} (u_{1}- u_{2}) ) \right]\d x\s \d t \\
&\qquad\qquad +\int_{\Omega \times [0, T]} (\square\, v_0) \s D^2_{\eps_1\eps_2}\big|_{\eps_1=\eps_2=0} (u_{1}-u_{2}) \d x\s\d t  \\
&\qquad  = -2 \int_{\Omega \times [0, T]} v_0 (a_1 -a_2)v_1\s v_2 \s \d x\s \d t   + \frac{1}{\epsilon_1\epsilon_2} \int_{\Omega \times [0, T]} v_0\, \square\s(\tildeR_1-\tildeR_2)\s \d x\d t.
\end{aligned}
\end{equation}
As an immediate consequence of this integral identity we obtain

\begin{equation}\label{eq:estimate_for_optimization}
\begin{split}
&2 \left| \langle v_0(a_1-a_2), v_1\s v_2 \rangle_{L^2(\Omega\times [0, T])} \right|\\
&\,  \leq  \left| \langle v_0, D^2_{\eps_1, \eps_2} (\Lambda_1-\Lambda_2) \left( \eps_1 f_1+ \eps_2 f_2 \right) \rangle_{L^2(\Sigma)} \right| +  \,\eps_1^{-1}\, \eps_2^{-1} \,\left| \langle v_0, \square\s(\tildeR_1-\tildeR_2)\rangle_{L^2(\Omega\times [0, T])} \right|\\
 &\,  \leq  4\, \eps_1^{-1}\, \eps_2^{-1}\left| \langle v_0, (\Lambda_1-\Lambda_2) \left( \eps_1 f_1+ \eps_2 f_2 \right) \rangle_{L^2(\Sigma)} \right| \\
 & \qquad +  \,\eps_1^{-1}\, \eps_2^{-1} \,\left| \langle v_0, \square\s(\tildeR_1-\tildeR_2)\rangle_{L^2(\Omega\times [0, T])} \right|\\
& \, \leq 4\, \delta \,\eps_1^{-1}\, \eps_2^{-1} \,  \norm{v_0}_{\dualH^{-r}(\Sigma)} + \eps_1^{-1}\, \eps_2^{-1}\, \norm{ \square\s(\tildeR_1-\tildeR_2)}_{E^{s+1}}\,\norm{v_0}_{\dualH^{-(s+1)}(\Omega\times [0, T])}\\
& \, \leq \eps_1^{-1}\, \eps_2^{-1} \,(\norm{v_0}_{\dualH^{-r}(\Sigma)}+ \norm{v_0}_{\dualH^{-(s+1)}(\Omega\times [0, T])})\\
&\qquad \quad  \times  \left(4\, \delta + C(s,T) (\norm{a_1}_{E^{s+1}}^2+\norm{a_2}_{E^{s+1}}^2) (\varepsilon_1 \Vert  f_1\Vert_{H^{s+1}(\Sigma)}+\varepsilon_2 \Vert  f_2\Vert_{H^{s+1}(\Sigma)})^3\right) \\
& \, \leq C \,\eps_1^{-1}\, \eps_2^{-1}  \left( \delta  + (\varepsilon_1 \Vert  f_1\Vert_{H^{s+1}(\Sigma)}+\varepsilon_2 \Vert  f_2\Vert_{H^{s+1}(\Sigma)})^3 \right),
\end{split}
\end{equation}
where we used the assumption $\norm{\Lambda_1(f)-\Lambda_2(f)}_{H^r(\widetilde{\Sigma})}\leq \delta$ and denoted
\[
C=  \max  \left\{ 4, C(s,T) (\norm{a_1}_{E^{s+1}}^2+\norm{a_2}_{E^{s+1}}^2)  \right\}  \,(\norm{v_0}_{H^{-r}(\Sigma)}+ \norm{v_0}_{H^{-(s+1)}(\Omega\times [0, T])}).
\]
A similar (slightly simpler) estimate can be achieved by using the assumption $$\abs{\langle v_0, \Lambda_1(f)-\Lambda_2 (f)\rangle_{L^2(\widetilde{\Sigma})}}\leq \delta.$$
Above we have used~\eqref{eq:estim_for_Rl} and that $E^{s+1}\subset H^{s+1}(\Omega\times [0,T])$ to bound the term $\square\s(\tildeR_1-\tildeR_2)$ in $E^{s+1}$. We have respectively denoted by $\dualH^{-r}(\Sigma)$ and $\dualH^{-(s+1)}(\Omega \times [0,T])$ the dual spaces of $H^{r}(\Sigma)$ and $H^{s+1}(\Omega \times [0,T])$ endowed with the following norms, see e.g. \cite{AF03}
\begin{align*}
\norm{w}_{\dualH^{-r}(\Sigma)}&:= \underset{v\in H^{r}(\Sigma),\, \norm{v}_{H^r(\Sigma)}\,\leq \,1}{\sup}\, \, |\langle v, w\rangle_{L^2(\Sigma)}|,\\
\norm{w}_{\dualH^{-(s+1)}(\Omega\times [0,T])}&:= \underset{v\in H^{s+1}(\Omega\times [0,T]),\, \norm{v}_{H^{s+1}(\Omega\times [0,T])}\, \leq \,1}{\sup}\,\, |\langle v, w\rangle_{L^2(\Omega\times[0,T])}|.
\end{align*}

\noindent\textbf{Step 2.} The second step is to suitably choose the functions $f_1$ and $f_2$ so that they allow us to obtain information about $a_1-a_2$ from the integral estimate \eqref{eq:estimate_for_optimization}. In this step, we shall need the following two technical results. 
\begin{lemma}\label{est:H_j}
Let $\alpha>0$, $\gamma\geq 0$ and $\tau\geq 1$. Let $\chi_\alpha\in C_c^\infty(\R)$ be a cut-off function supported on $[-\alpha,\alpha]$ , $\abs{\chi_\alpha}\leq 1$. Consider the function $H\in C_c^{\infty}(\mathbb{R})$ defined by
\[
H(l)=\chi_\alpha(l)\tau^{1/2} \e^{-\frac{1}{2}\tau\, l^2}, \quad l\in \R.
\]
Let $(x_0, t_0)\in \mathbb{R}^2$ and define 
\begin{equation}
\begin{split}
H_1^{\tau,(x_0,t_0)}(x,t)&:=H\big((x-x_0)-(t-t_0)\big), \\
H_2^{\tau,(x_0,t_0)}(x,t)&:= H\big((x-x_0)+(t-t_0)\big). 
\end{split}
\end{equation}
The following estimate holds
\[
\norm{H_1^{\tau,(x_0,t_0)}}_{H^{\gamma}(\Sigma)}+ \norm{H_2^{\tau,(x_0,t_0)}}_{H^{\gamma}(\Sigma)}\leq C \,\tau^{\frac{\gamma +1}{2}}.
\]
The constant $C$ is independent of $(x_0,t_0)\in \R^2$.
\end{lemma}
\begin{proof}
 Let $(x_0,t_0)\in \R^{2}$ and  $\beta_1,\beta_2\in \N$. Let us write 
 \[
F(x,t)= H\big((x-x_0)-(t-t_0)\big).  
 \]
We have for all $\tau\geq 1$ that 
\begin{equation}
 \begin{split}
  &\norm{\p_{x}^{\beta_1} \p_t^{\beta_2} F}_{L^2(\Omega\times [0,T])}^2=\tau  \int_{\Omega}\int_0^T\left[\p_{x}^{\beta_1} \p_t^{\beta_2}\Big(\chi_\alpha(x-x_0+ t- t_0) \e^{-\frac{\tau}{2} (x-x_0+ t- t_0)^2}\Big)\right]^2 \d t  \d x\\
&\leq C\s \tau \tau^{2(\beta_1+\beta_2)}\int_{\Omega}\int_0^T|\chi_\alpha(x-x_0+ t- t_0)|^2\, (x-x_0+ t- t_0)^{2(\beta_1+\beta_2)}\e^{-\tau (x-x_0+ t- t_0)^2} \d t \, \d x\\
&=C\s \tau\tau^{2(\beta_1+\beta_2)} \int_{\Omega} \int_{x-x_0-t_0}^{x-x_0+T-t_0} |\chi_\alpha(h)|^2\, h^{2(\beta_1+\beta_2)}\e^{-\tau h^2} \d h \, \d x \\
&\leq C\s \tau\tau^{2(\beta_1+\beta_2)} \int_{\Omega} \int_{-\infty}^{\infty} |\chi_\alpha(h)|^2\, h^{2(\beta_1+\beta_2)}\e^{-\tau h^2} \d h \, \d x \\
&\leq C \tau\tau^{2(\beta_1+\beta_2)}\tau^{-(\beta_1+\beta_2)-1/2}\int_\Omega \d x=C_\Omega\tau^{(\beta_1+\beta_2)+1/2}.
 \end{split}
\end{equation}
%
Here in the second line we used the fact that the largest power of $\tau$ in the calculation happens when all the derivatives hit the exponential and none the cut-off function. Therefore, when $\tau\geq 1$, we may absorb the other terms implicit in the calculation to the constant $C$.
We also made a change of variables
\[
 h=x-x_0+ t- t_0
\]
in the integral in the variable $t$, while considering $x$ is fixed. We also used 
\[
\int_\R h^{2(\beta_1+\beta_2)}e^{-\tau h^2}\d h\sim \tau^{-(\beta_1+\beta_2)-1/2}.
\]
Thus, we have
\[
 \norm{F}_{H^{\beta_1+\beta_2}(\Omega\times [0,T])}^2\leq C_\Omega \tau^{(\beta_1+\beta_2)+1/2}.
\]

By a standard interpolation argument between Sobolev spaces, see for instance \cite[Theorem 6.2.4/6.4.5]{BL76}, we then obtain for all $\gamma\geq 0$ that
\begin{equation}
\begin{aligned}
\norm{F}_{H^\gamma(\Omega\times [0,T])}^2& \leq C \tau^{\gamma+1/2}.
\end{aligned}
\end{equation}
Finally, by using the trace theorem we have
\[
\norm{F}_{H^\gamma(\Sigma)}^2\leq C \norm{F}_{H^{\gamma+1/2}(\Omega\times [0,T])}^2\leq  
C\tau^{\gamma + 1}. 
\]
Similar argument yields the same estimate for $H_2^{\tau, (t_0, x_0)}$. 
This completes the proof.
\end{proof}

\begin{lemma}\label{est:tau}
Let $b\in C^{1}_c(\mathbb{R}^2)$ and $\tau>0$. The following estimate  
\[
\left| b(x_0,t_0)-  \frac{\tau}{\pi} \int_{\mathbb{R}^2}b(x,t)\e^{-\tau((x-x_0)^2 + (t-t_0)^2)} \d x\, \d t \right| \leq \frac{\sqrt{\pi}}{2} \left\|b \right\|_{C^1}\tau^{-1/2}
\]
holds true for all $(x_0, t_0)\in \mathbb{R}^2$. In particular, the integral on the left converges uniformly to $b$ when $\tau \to \infty$.
\end{lemma}
\begin{proof}
Without loss of generality we prove the estimate when $(x_0, t_0)=(0,0)$, because it can be later applied to $b(x+x_0, t+t_0)$ in place of $b(x,t)$. Using polar coordinates, one can see that $\int_{\mathbb{R}^2}\e^{-(x^2+t^2)} \d x \, \d t = \pi$ and $\int_{\mathbb{R}^2}2\, \sqrt{x^2+t^2}\,\e^{-(x^2+t^2)} \d x \, \d t=\pi^{3/2}$, and noting that $\left|  b(0,0)- b(\tau^{-1/2}x,\tau^{-1/2}t)\right| \leq \left\|b \right\|_{C^1}\tau^{-1/2}\left| (x,t)\right|$ for all $(x,t)\in \mathbb{R}^2$, we immediately deduce
\begin{align*}
&\left| b(0,0)-   \frac{\tau}{\pi} \int_{\mathbb{R}^2}b(x,t)\e^{-\tau(x^2 + t^2)} \d x\, \d t \right|\\
&\qquad = \left| b(0,0)-   \frac{1}{\pi} \int_{\mathbb{R}^2}b(\tau^{-1/2}x,\tau^{-1/2}t)\e^{-(x^2 + t^2)} \d x\, \d t \right| \\
&\qquad  = \left|  \frac{1}{\pi} \int_{\mathbb{R}^2}\left(b(0,0)-b(\tau^{-1/2}x,\tau^{-1/2}t)\right)\e^{-(x^2 + t^2)} \d x\, \d t \right|\\
&\qquad  \leq \frac{\tau^{-1/2}}{\pi}  \left\|b \right\|_{C^1} \int_{\R^2} |(x,t)|e^{-(x^2+t^2)} \d x\, \d t    = \frac{\sqrt{\pi}}{2} \left\|b \right\|_{C^1}\tau^{-1/2}. 
\qedhere
\end{align*}
\end{proof}

 Let $(x_0, t_0)\in \supp(a_j)$. Note that  
 \[
  \square H_j=0.
 \]
We choose 
\begin{equation}\label{choice_of_vj}
 v_j=H_j \text{ and } f_j=H_j|_{\Sigma}, \quad j=1,2,
\end{equation}
where $H_j=H_j^{\tau,(x_0,t_0)}$ is as in Lemma \ref{est:H_j} with $\gamma=s+1$ and the cut-off function $\chi_\alpha$  so that $\chi_\alpha(0)=1$. 
We assume that $\alpha>0$ is small enough that $f_j$ vanishes near $\{t=0\}$, so that $\p_t^k\big|_{t=0} f_j=0$, $k=1,\ldots,s$, on $\p\Omega$.
Substituting this choice of $v_j$ into inequality \eqref{eq:estimate_for_optimization}, and using Lemma \ref{est:tau} with 
 \[
  b(x,t):=v_0\s(a_1-a_2)\chi_\alpha(x-x_0- (t-t_0))\chi_\alpha(x-x_0+ (t-t_0)),
 \]
we get
 \begin{equation}\label{est:optimising}
 \begin{split}
&\left| (v_0(a_1-a_2))(x_0, t_0)\right|  \leq   \frac{1}{\pi}  \left| \int_{\Omega\times [0,T]} v_0 (a_1 -a_2)H_1\, H_2 \,\d x\,\d t \right|\\ 
&    \qquad  + \left| (v_0(a_1-a_2))(x_0, t_0) -\frac{1}{\pi} \int_{\Omega\times [0,T]} v_0 (a_1 -a_2)H_1\, H_2 \,\d x\,\d t \right| \\
& \quad\quad     \qquad  \leq  C_{\Omega, T, a_j, \chi_\alpha}  \left(2\tau^{-1/2} + \frac{\delta}{2} \eps_1^{-1}\eps_2^{-1}+\eps_1^{-1}\eps_2^{-1}\, (\eps_1+ \eps_2)^3  \,\tau^{\frac{3}{2}s + 3} \right) \left\| v_0 \right\|_{C^1}\\
& \quad\quad    \qquad  \leq  \frac{C_{\Omega, T, a_j, \chi_\alpha}M}{\kappa^3}  \left(2\tau^{-1/2} + \frac{\kappa^3\delta}{2M} \eps_1^{-1}\eps_2^{-1}+\epsilon_1^{-1}\eps_2^{-1}\, (\eps_1+ \eps_2)^3 \,\tau^{\frac{3}{2}s + 3} \right) \left\| v_0 \right\|_{C^1}.
\end{split}
\end{equation}

In the last step we scaled $\delta$ by a constant $\kappa^3/M$, which we without loss of generality assume is $<1$.
This scaling is purely technical and will be clarified in Lemma~\ref{lemma:final_estimate}.

\noindent\textbf{Step 3.}
 Our last step is optimizing $\tau$, $\eps_1$  and $\eps_2$ in terms of $\delta$ to get the right hand side of~\eqref{est:optimising} as small as possible. The constants $2$ and $1/2$ in front of $\tau^{-1/2}$ and $\delta\epsilon^{-2}$ are used only to simplify the formulas. We begin by setting 
 \[
 \epsilon_1=\epsilon_2=\epsilon.
 \]
Note that we have
\begin{equation}\label{eq:epsHj}
 \eps\norm{f_j}_{H^{s+1}(\Sigma)}\sim \eps\tau^{\frac{s+2}{2}}.
\end{equation}
To guarantee the unique solvability of the non-linear wave equation, we require the quantity on the right-hand side of \eqref{eq:epsHj} is bounded by $\kappa$ as in Lemma \ref{lemma:nonlinear-solutions}.
The following Lemma~\ref{lemma:final_estimate} shows how to optimally choose the parameters $\lambda$ and $\eps$ of the inverse problem given a priori bounds $\kappa$ and $\delta$ of the forward problem, while keeping the size of the sources $\eps_jf_j$ small.

\begin{lemma}\label{lemma:final_estimate}
For any given $\delta\in(0,M)$ and $\kappa\in(0,1)$ small enough we find $\eps(\delta,\kappa)=\eps$ and $\tau(\delta,\kappa)=\tau\geq 1$  such that
\begin{equation}\label{eq:final_estimate}
f(\epsilon, \tau):= 2\tau^{-1/2} + \frac{\kappa^3\delta}{2M} \epsilon^{-2}+\epsilon \,\tau^{\frac{3}{2}s + 3}  \leq C_{s,M,\kappa} \, \delta^{\frac{1}{  6s+15}}
\end{equation}
and we also have
$$
\eps\tau^{\frac{s+2}{2} }\leq \kappa.
$$
The constant $C_{s,M,\kappa}$ is independent of $\delta$.
\end{lemma}

\begin{proof}
To simplify notation, let $\widehat{s}:=3s/2+3$ and $\gamma_0=\kappa^3/M$. A direct computation shows that
\begin{equation}
\partial_{\epsilon}f= -(\gamma_0\delta) \epsilon^{-3}+ \tau^{\hat{s}}, \quad \partial_\tau f=- \tau^{-3/2}+ \epsilon \hat{s}\tau^{\hat{s}-1}.
\end{equation}
Making $\partial_{\epsilon}f=\partial_{\tau}f=0$, we obtain the critical points of $f$, namely
\begin{equation}\label{eq:critical1d}
\tau= {\widehat{s}}^{\s\s\,  -\frac{6}{4\widehat{s}+3}} (\gamma_0\delta)^{-\frac{2}{4\widehat{s}+3}}, \quad \epsilon={\widehat{s}}^{\s\s\frac{2\widehat{s}}{4\widehat{s}+3}} (\gamma_0\delta)^{\frac{2\widehat{s}+1}{4\widehat{s}+3}}.
\end{equation}
With these choices of $\tau$ and $\epsilon$, one can check that $\tau^{-1/2}$, $(\gamma_0\delta) \epsilon^{-2}$ and $\epsilon \tau^{\widehat{s}}$ are all bounded by $C_s\, (\gamma_0\delta)^{\frac{1}{4\widehat{s}+3}}$. Also, $\tau\geq 1$ for $\kappa$ small enough.

Furthermore, since
$$
\eps\tau^{\frac{\widehat{s}}{3}} = (\gamma_0\delta)^{1/3},
$$
we have that
$$
(\eps\tau^{\frac{s+2}{2}})
 \leq \kappa
$$
for any $0<\delta<M$.
This finishes the proof. 
\end{proof}

Equation \eqref{eq:critical1d} in the proof of Lemma \ref{lemma:final_estimate} also shows how to choose the parameters $\tau$ and $\eps$ depending on $\delta$ and $\kappa$. We also see that $\eps\norm{f_j}_{H^{s+1}(\Sigma)}\leq \kappa$.

Continuing from \eqref{est:optimising} by putting $\eps_1=\eps_2=\eps$ and then applying Lemma \ref{lemma:final_estimate} we finally obtain
\begin{align}\label{optimized}
&\left| (v_0(a_1-a_2))(x_0, t_0)\right|  \\
&\qquad\leq 
\frac{C_{\Omega, T, a_j, \chi_\alpha}M}{\kappa^3}  \left(2\tau^{-1/2} + \frac{\kappa^3\delta}{2M} \eps_1^{-1}\eps_2^{-1}+\eps_1^{-1}\eps_2^{-1}\, (\eps_1+ \eps_2)^3 \,\tau^{\frac{3}{2}s + 3} \right) \left\| v_0 \right\|_{C^1}\\
&\qquad=
 \frac{C_{\Omega, T, a_j, \chi_\alpha}M}{\kappa^3}  \left(2\tau^{-1/2} + \frac{\kappa^3\delta}{2M} \eps^{-2}+\eps \,\tau^{\frac{3}{2}s + 3} \right) \left\| v_0 \right\|_{C^1}
\leq 
C
\delta^{\frac{1}{6s+15}}. 
\end{align}
Recall that $v_0$ satisfies \eqref{eq:defv_0}.
%
In particular $v_0(x_0,t_0)=1$.
This finishes the proof of Theorem \ref{thm:stability}. 
Moreover, by letting $\delta\to 0$ we obtain Theorem~\ref{thm:uniqueness}.
\qed


\subsubsection{Proof of Theorem \ref{thm:noise} in $1+1$ dimensions with $m=2$}
We prove Theorem \ref{thm:noise} in the case $1+1$ dimensions with $m=2$. The proof follows from similar arguments we used in the previous section. 
Let us consider any point $(x_0,t_0)\in\supp(a)$ and let $v_0$, $H_1$ and $H_2$ be as in~\eqref{eq:defv_0} and Lemma~\ref{est:H_j} respectively. Let us also set $v_j=H_j$, $f_j=H_j|_{\Sigma}$, $j=1,2$, as before. Then, as in the proof of Theorem \ref{thm:stability} in $1+1$ dimensions with $m=2$, we have 
\begin{align*}
&\left|
-2v_0a(x_0,t_0) - \frac{1}{\pi}D_{\eps_1,\eps_2}^2\big|_{\eps_1=\eps_2=0} \int_\Sigma v_0
(\Lambda+\E)(\eps_1f_1 + \eps_2f_2)\d S
\right|\\
&\quad\leq
\left|
-2v_0a(x_0,t_0) + \frac{2}{\pi}\int_{\Omega\times[0,T]}v_0 a\s v_1 v_2\d x\,\d t\right|\\
&\qquad+
\left|
\int_{\Omega\times[0,T]} v_0\frac{1}{\eps_1\eps_2}\square\tildeR\d x\,\d t
\right|
+
\left|
D_{\eps_1,\eps_2}^2\big|_{\eps_1=\eps_2=0} \int_\Sigma v_0\s \E(\eps_1f_1 + \eps_2f_2)\d S
\right|\\
&=:I_1 + I_2 + I_3.
\end{align*}
%
By using Lemma \ref{est:tau} on the first term $I_1$ we obtain
\[
I_1 \leq C_{\Omega,T,a} \tau^{-1/2}.
\]
The third term is estimated simply by
\[
I_3 \leq C_{\Omega,T,a} \frac{\delta}{\eps_1\eps_2} \Vert v_0\Vert_{\dualH^{-r}(\Sigma)}.
\]
The remaining term $I_2$ can be estimated by using \eqref{est:square_R} as
\begin{align}
I_2 &\leq \frac{C_{\Omega,T,a}}{\eps_1\eps_2}\Vert \square\tildeR\Vert_{E^{s+1}}
\Vert v_0\s\Vert_{\dualH^{-s-1}(\Sigma)}\\
&\leq
\frac{C_{\Omega,T,a}}{\eps_1\eps_2}\left(
\eps_1\Vert H_1\Vert_{H^{s+1}(\Sigma)} + \eps_2\Vert H_2\Vert_{H^{s+1}(\Sigma)}
\right)^3 \Vert v_0\Vert_{\dualH^{-(s+1)}(\Sigma)}\\
&\leq
\frac{C_{\Omega,T,a}}{\eps_1\eps_2}(\eps_1+\eps_2)^3\tau^{3s/2+3} \Vert v_0\Vert_{\dualH^{-(s+1)}(\Sigma)}.
\end{align}
Combining everything and changing the constant if necessary, we have that
\begin{equation}
I_1+I_2+I_3\leq C_{\Omega,T,a} \left(2\tau^{-1/2} + \frac{\kappa^3\delta}{2M} \epsilon_1^{-1}\epsilon_2^{-1}+\epsilon_1^{-1}\epsilon_2^{-1}\, (\eps_1+ \eps_2)^3 \,\tau^{\frac{3}{2}s + 3} \right),
\end{equation}
which is the same estimate as in equation \eqref{est:optimising}.
Choosing now $\eps_1=\eps_2=\eps$ and optimizing by using Lemma \ref{lemma:final_estimate} we have the claimed estimate.
This completes the proof of Theorem \ref{thm:noise}.\qed

\section{Proofs of the main results in dimensions $n+1$, $n\geq 2$, and $m\geq 2$}\label{sec:main-results-nD}
Here we finish the proof of Theorem~\ref{thm:stability}. To do that, we need to first discuss Radon transformation.

\subsection{Radon transform} \label{RT}
Let $f$ be a function on $\mathbb{R}^n$, which is integrable on each hyperplane in $\mathbb{R}^n$. Each hyperplane can be expressed as the set of solutions $x$ to the equation $x\cdot \theta =\eta$, where $\theta 
\in S^{n-1}$ is the unit normal of the hyperplane, and $\eta\in \mathbb{R}$. The Radon transform of $f$ is defined by
\begin{equation}\label{def:Radon_transform}
(\pmb{R}f)(\theta, \eta)= \int_{x\cdot \theta =\eta} f(x)\s \d x=\underset{y\in\s \theta^{\perp}}{\int} f(\eta\s\theta +y) \s \d y.
\end{equation}
Here $\theta^{\perp}$ denotes the set of orthogonal vectors to $\theta$. We remark that if a function is supported in a ball of radius $M$ in $\R^n$, then its Radon transformation is supported in its $\eta$ variable in $[-M,M]$. 

There is a natural relation between $f$ and its Radon transform on the Fourier side. This is usually called the \textit{Fourier slice theorem}, see for instance \cite[Theorem 1.1]{Na}.
\begin{proposition*}[Fourier slice theorem]\label{th:Fourier_slice}
Let $\theta\in S^{n-1}$. For $f\in C_c^\infty(\R^n)$ we have
\[
\F_{\eta\to\sigma}\big((\pmb{R} f)(\theta, \eta)\big)(\sigma)= (2\pi)^{\frac{n-1}{2}}\widehat{f}(\sigma\theta), \quad \sigma\in \R.
\]
Here $\F_{\eta\to\sigma}$ denotes the one dimensional Fourier transform with respect to $\eta$ and the hat-notation $\widehat{f}$ is used to denote the $n$-dimensional Fourier transform.
More precisely, 
\[
\F_{\eta\to\sigma}\big((\pmb{R} f)(\theta, \eta)\big)(\sigma)= \int_\R e^{-i\eta\sigma} (\pmb{R}f)(\theta, \eta)\s \d \eta, \quad \widehat{f}(\xi)= \int_{\R^n} e^{-i x\cdot \xi} f(x) \s \d x.
\]
\end{proposition*}

Using the Fourier slice theorem we can show that the Sobolev $H^{-\beta}$ norm of a function can be estimated by the $L^2$ norm of its Radon transform, if the Sobolev index is $\beta \geq (n-1)/2$. This is a special case of \cite[Theorem 5.1]{Na}. We give a proof of Lemma \ref{RT:pointwise} for the convenience of the reader in Appendix \ref{app:B}.

\begin{lemma}\label{RT:pointwise}
Let $\beta\geq (n-1)/2$. Let $f\in C^\infty_c(\R^n)$ with $\supp f \subset B_M(0)\subset \R^n$ for some $M>0$. Consider $F\in L^2(S^{n-1}\times [-M, M])$ and assume that there exists a constant $C_0>0$  such that
\[
|(\pmb{R}f)(\theta, \eta)|\leq C_0 F(\theta, \eta), \quad  \textit{a.e } (\theta,\eta)\in S^{n-1}\times [-M, M]. 
\] 
Then we have the following estimate
\[
\norm{f}_{H^{-\beta}(\R^n)}\leq (2\pi)^{1/2}\,C_0\, \norm{F}_{L^2(S^{n-1}\times [-M,M])}.
\]
Here $C_0$ is independent of $\theta$ and $\eta$.
\end{lemma}

\begin{lemma}\label{estimate:H_jn}
Let $\alpha>0$, $\gamma\geq 0$ and $\tau\geq 1$. Let $\chi_\alpha\in C_c^\infty(\R)$ be a cutoff function supported on $[-\alpha,\alpha]$, $\abs{\chi_\alpha}\leq 1$. Consider $H\in C_c^{\infty}(\mathbb{R})$ defined by
\[
H(l)=\chi_\alpha(l)\tau^{1/2} \e^{-\frac{1}{2}\tau\, l^2}.
\]
In addition, consider $t_0\in \R$, $\eta\in \R$ and $\theta\in S^{n-1}$, and define 
\begin{align*}
H_1^{\tau, (t_0, \theta,\eta)}(x,t)&:=H(x\cdot \theta -t - (\eta-t_0)),\\
H_2^{\tau, (t_0, \theta,\eta)}(x,t)&:=H(- x\cdot \theta -t + (\eta+t_0)).
\end{align*}
The following estimate holds 
\[
\norm{H_1^{\tau, (t_0,\theta, \eta)}}_{H^{\gamma}(\Sigma)}+ \norm{H_2^{\tau, (t_0, \theta,\eta)}}_{H^{\gamma}(\Sigma)}\leq 
C\tau^{\frac{\gamma +1}{2}},
\]
where the implicit constant is independent of $t_0$, $\theta$ and $\eta$. 
\end{lemma}
The proof of this lemma is similar to Lemma \ref{est:H_j} and can be found in Appendix \ref{app:B}.

Let $\Omega\subset \R^n$. We write $\mathscr{R}(G)$ for the \emph{partial Radon transformation} of a function $G=G(x,t)\in \Omega\times\R$, in its spatial variable $x$:
\[
 \mathscr{R}(G)(t,\theta,\eta)=\int_{x\cdot \theta =\eta} G(x,t)\d x, \quad \theta\in S^{n-1}, \ \eta\in \R.
\]

\begin{lemma}\label{est:tau_ndimension}
Let $G\in C^{\infty}_c(\mathbb{R}^{n+1})$. Let $t_0\in \R$ and $\tau>0$. There exists $C>0$ (depending only on $\supp G$) such that the following estimate 
\begin{align*}
&\left| \mathscr{R}(G)(t_0,\theta, \eta)-  \frac{\tau}{\pi} \int_{\R}\int_{\mathbb{R}^{n}}G(x,t)\e^{-\tau((x\cdot \theta-\eta)^2 + (t-t_0)^2)} \d x\, \d t \right|\\
& \leq \frac{\sqrt{\pi}}{2} C \left\| G \right\|_{C^1(\R^{n+1})}\tau^{-1/2}
\end{align*}
holds. Here $C$ is independent of $\theta\in S^{n-1}$ and $\eta\in \R$.
\end{lemma}

\begin{proof}
Let $\theta\in S^{n-1}$ and $\eta\in \R$. We write any $x\in \R^n$ as
\[
x= s'\, \theta +y, \quad s'= x\cdot \theta \,\, \text{and} \,\, y\in \theta^\perp.
\]
By making the change of variables $x \mapsto (y,s')$ with $\d x= \d y\,\d s'$, we obtain
\begin{align*}
   & \int_{\R}\int_{\mathbb{R}^{n}}G(x,t)\,\e^{-\tau((x\cdot \theta-\eta)^2 + (t-t_0)^2)} \d x\, \d t \\
   &= \int_{\R} \left(\int_{\R} \int_{y\in\s\theta^\perp} G(s'\, \theta +y,t) \e^{-\tau((s'-\eta)^2 + (t-t_0)^2)} \d y \, \d s'   \right)\d t\\
   & = \int_{\R} \left(\int_{\R}\e^{-\tau((s'-\eta)^2 + (t-t_0)^2)}  \int_{y\in\s\theta^\perp}G(s'\, \theta +y,t)  \d y \, \d s'   \right)\d t\\
   & = \int_{\R^2} \mathscr{R}(G)(t,\theta, s')\, \e^{-\tau((s'-\eta)^2 + (t-t_0)^2)} \d s'\, \d t.
\end{align*}
The result will follow by applying Lemma \ref{est:tau} if we can show that
\begin{align*}
&\underset{(s',t)\in \R^2}{\sup} | (\mathscr{R}(G)(t,\theta, s')|+ \underset{(s',t)\in \R^2}{\sup} |\p_{s'} \mathscr{R}(G)(t,\theta, s')|+  \underset{(s',t)\in \R^2}{\sup} |\p_t \mathscr{R}(G)(t,\theta, s')|:= L(\theta)
\end{align*}
is uniformly bounded in $\theta\in S^{n-1}$.
But this follows, since $G$ is compactly supported and smooth so that by Fubini's theorem we can change the order of differentiation and integration.
%
\end{proof}

\subsection{Proof of Theorem \ref{thm:stability} in $n+1$ dimensions with $m\geq 2$}

The proof is quite similar to the one in $1+1$ dimension with $m=2$. The main difference between the proofs is that instead of having a  pointwise estimate of the function $v_0(a_1-a_2)$, see~\eqref{est:optimising}, we obtain estimates for the partial Radon transformation of this function when $n\geq 2$. 
Here $v_0$ satisfies $\square v_0=0$ as before, see~\eqref{eq:defv_0}. 

We have the integral identity~\eqref{eq:integral_identity_finite_difference}
\begin{equation*}
\begin{aligned}
  -m!\int_{\Omega\times [0,T]} a\s v_0\s v_1\s v_2\cdots v_m\s \d x \s \d t 
  &=
\int_{\Sigma}v_0\s D_{\eps_1,\ldots,\eps_m}^m\big|_{\eps=0}\Lambda(\eps_1f_1+\cdots+\eps_mf_m)\s \d S\\&\qquad + \frac{1}{\eps_1\eps_2\cdots\eps_m}\int_{\Omega\times [0,T]} v_0\square\s \widetilde{\mathcal{R}}\s \d x \s \d t .
\end{aligned}
\end{equation*}
It follows that we have an estimate similar to~\eqref{eq:estimate_for_optimization}
\begin{equation}
\begin{split}
&m! \left| \langle v_0(a_1-a_2), v_1\s \cdots \s v_m \rangle_{L^2(\Omega\times [0, T])} \right|\\
& \, \leq C \,\eps_1^{-1}\cdots \eps_m^{-1}  \left( \delta  + (\varepsilon_1 \Vert  f_1\Vert_{H^{s+1}(\Sigma)}+\cdots + \varepsilon_m \Vert  f_m\Vert_{H^{s+1}(\Sigma)})^{2m-1} \right),
\end{split}
\end{equation}
where 
\[
C=  \max  \left\{ 4, C(s,T) (\norm{a_1}_{E^{s+1}}^2+\norm{a_2}_{E^{s+1}}^2)  \right\}  \,(\norm{v_0}_{H^{-r}(\Sigma)}+ \norm{v_0}_{H^{-(s+1)}(\Omega\times [0, T])}).
\]
%

We first choose the boundary values as follows. Let $t_0\in \R$, $\eta\in \R$ and $\theta\in S^{n-1}$. 
For $j=1,2$, we choose 
\begin{equation}
 v_j=H_j \text{ and } f_j=H_j|_{\Sigma}, \quad j=1,2,
\end{equation}
where $H_j=H_j^{\tau,(t_0,\theta,s)}$ are as in Lemma \ref{estimate:H_jn} with $\gamma=s+1$ and the cutoff function $\chi_\alpha$  so that $\chi_\alpha(0)=1$.
We assume that $\alpha>0$ is small enough that $f_j$ vanishes near $\{t=0\}$, so that $\p_t^k\big|_{t=0} f_j=0$, $k=1,\ldots,s$, on $\p\Omega$.

For $j=3,\ldots,m$, we let $\tau_0>0$ and we choose
\begin{equation}\label{choice_of_vj_n}
 v_j=\tau_0^{-1/2}H_1^{\tau_0,(t_0,\theta,\eta)} \text{ and } f_j=\tau_0^{-1/2}H_1^{\tau_0,(t_0,\theta,\eta)}|_{\Sigma}.
\end{equation}
Let us write
\[
 \overline{v}=v_0v_3\cdots v_m.
\]
Note that $\overline{v}(x,t_0)=1$ if $x\cdot \theta=\eta$.

Substituting these choices of $v_j$ into inequality \eqref{eq:estimate_for_optimization}, and using Lemma \ref{est:tau_ndimension} with 
\[
G(x,t)= \overline{v}(x, t_0)(a_1-a_2)(x, t_0)\chi_\alpha(x\cdot \theta -t - (\eta-t_0))\chi_\alpha(-x\cdot \theta -t + (\eta+t_0))
\]
we get
 \begin{equation}\label{est:optimising_n}
 \begin{split}
&\left|\mathscr{R}(G)(t_0,\theta,\eta)\right| \leq   \frac{1}{\pi}  \left| \int_{\Omega\times [0,T]} \overline{v} (a_1 -a_2)H_1\, H_2 \,\d x\,\d t \right|\\ 
&    \qquad  + \left| \mathscr{R}(G)(t_0,\theta,\eta)-\frac{1}{\pi} \int_{\Omega\times [0,T]} \overline{v} (a_1 -a_2)H_1\, H_2 \,\d x\,\d t \right| \\
&   \leq  C\left(2\tau^{-1/2} + \epsilon_1^{-1}\cdots\epsilon_m^{-1}\Big(\delta +(\eps_1+\cdots +\eps_m)^{2m-1} \,(\tau^{\frac{s+2}{2}})^{2m-1}\Big) \right) \left\| \overline{v} \right\|_{C^1}.
\end{split}
\end{equation}
Since $\overline{v}(t_0,\theta,\eta)=1$, we have by the definition of the Radon transform that
\begin{equation}\label{ov_constan_on_slice}
 \begin{split}
     \mathscr{R}(G)(t_0,\theta,\eta)&= \int_{x\cdot \theta =\eta}G(x,t_0) \d x\\
     &= \int_{x\cdot \theta =\eta} \overline{v}(x,t_0)(a_1-a_2)(x,t_0) \chi_\alpha (0) \chi_\alpha (0) \d x\\
    & = \int_{x\cdot \theta =\eta} (a_1-a_2)(x, t_0) \d x=  \mathscr{R}(a_1-a_2)(t_0,\theta, \eta).
    \end{split}
    \end{equation}
By using this identity and \eqref{est:optimising_n} we obtain
\begin{align*}
&\left|\mathscr{R}(a_1-a_2)(t_0,\theta,\eta)\right| \\
&       \leq  \widetilde{C}  \left(2\tau^{-1/2} + \epsilon_1^{-1}\cdots\epsilon_m^{-1}\Big(\delta + (\eps_1+ \cdots+\eps_m)^{2m-1} \,(\tau^{\frac{s+2}{2}})^{2m-1}\Big) \right) \left\| \overline{v} \right\|_{C^1} \\
&=:C(\eps_j,\tau,\delta).
\end{align*}
Let us choose $F\in L^2(S^{n-1}\times[-M,M])$, $F\equiv 1$, where $\mathrm{supp}(a_j)\subset B_M(0)$ for $j=1,2$.
Applying Lemma \ref{RT:pointwise} with $f(\ccdot)=(a_1-a_2)(\ccdot,t_0)$ and $\beta=(n-1)/2$, we obtain
\[
 \norm{(a_1-a_2)(\ccdot,t_0)}_{H^{-(n-1)/2}(\R^n)}  \leq (2\pi)^{1/2}\,C(\eps_j,\tau,\delta)\,C_{\supp(a_j)}.
\]

By the a priori assumption on the potentials (admissible potentials), we know that
\[
\norm{(a_1-a_2)(\ccdot,t_0)}_{H^{s+1}(\R^n)}\leq 2L.
\] 
Hence, we have estimates for $(a_1-a_2)(\ccdot,t_0)$ in the Sobolev spaces $H^{-(n-1)/2}(\R^{n})$ and $H^{s+1}(\R^{n})$. Since $s+1>(n+1)/2$, we deduce that $-l (n-1)/2+ (1-l)(s+1)>n/2$ for all $l\in (0, 1/(2n)]$.
Using an interpolation argument, see for instance \cite[Theorem 6.2.4 or 6.4.5]{BL76}, and the Sobolev embedding $H^{\gamma}(\R^n)\subset L^\infty(\R^n)$, $\gamma>n/2$, and after changing constants if needed, we obtain
\begin{align*}
&\norm{(a_1-a_2)(\ccdot,t_0)}_{L^\infty(\R^n)} \leq c_{n,s,l}\norm{(a_1-a_2)(\ccdot,t_0)}_{H^{-(n-1)l/2+ (s+1)(1-l)}(\R^n)}\\
&\leq c_{n,s,l}\norm{(a_1-a_2)(\ccdot,t_0)}_{H^{-(n-1)/2}(\R^n)}^l\norm{(a_1-a_2)(\ccdot,t_0)}_{H^{s+1}(\R^n)}^{1-l}\\
& \leq \widetilde{C}  \left(2\tau^{-1/2} + \epsilon_1^{-1}\cdots\epsilon_m^{-1}\Big(\delta + (\eps_1+\cdots+ \eps_m)^{2m-1} \,(\tau^{\frac{s+2}{2}})^{2m-1}\Big) \right)^l\left\| \overline{v} \right\|_{C^1}^l. 
\end{align*}
 The above estimate is uniform in $t_0\in \R$, therefore
\begin{multline*}
\norm{a_1-a_2}_{L^\infty(\R^{n+1})}\\ \leq  \widetilde{C}  \left(2\tau^{-1/2} + \epsilon_1^{-1}\cdots\epsilon_m^{-1}\Big(\frac{\kappa^{2m-1}\delta}{mM} + \frac{(\eps_1+\cdots+ \eps_m)^{2m-1}}{m-1} \,(\tau^{\frac{s+2}{2}})^{2m-1} \Big)\right)^l\left\| \overline{v} \right\|_{C^1}^l
\end{multline*}
for all $l\in (0, 1/(2n)]$. Here we made a similar scaling of $\delta$ as in equation \eqref{est:optimising}.

As before, we choose $\eps_1=\ldots=\eps_m=\eps$.
Then the last step is to optimize in $\tau$ and $\eps$ as was done in Lemma \ref{lemma:final_estimate}.
We recall, that the quantity
$$
\eps\norm{f_j}_{H^{s+1}(\Sigma)} \sim \eps\tau^{\frac{s+2}{2}}
$$
should be bounded by $\kappa$.
The next lemma generalizes Lemma \ref{lemma:final_estimate}:

\begin{lemma}\label{lemma:final_estimate_n}
For any given $\delta\in(0,M)$ and $\kappa\in(0,1)$ small enough  we find $\eps(\delta,\kappa)=\eps$ and $\tau(\delta,\kappa)=\tau\geq 1$ such that
\begin{equation*}
f(\epsilon, \tau):= 2\tau^{-1/2} + \frac{\kappa^{2m-1}\delta}{mM} \epsilon^{-m}+\frac{1}{m-1}\epsilon^{m-1} \,\tau^{\frac{s+2}{2}(2m-1)}  \leq C_{s,M,\kappa} \, \delta^{\frac{m-1}{(2m-1)(m(s+2)+1) }}
\end{equation*}
and we also have
$$
\eps\tau^{\frac{s+2}{2}} \leq \kappa.
$$
The constant $C_{s,M,\kappa}$ is independent of $\delta$.
\end{lemma}
We leave the proof of this Lemma to Appendix \ref{app:B}. Finally, we choose the interpolation parameter $l=1/(2n)$. This finishes proof of Theorem \ref{thm:noise_Rn}. \qed

Finally, we formulate a reconstruction result from noisy measurements in $\R^{n+1}$, $n\geq 2$, in terms of the Radon transformation. This result is analogous to Theorem \ref{thm:noise}.

\begin{proposition}[Stability with noise and reconstruction $n\geq 2$]\label{thm:noise_Rn}
Let $\Omega\subset \R^n$ be a bounded domain with a smooth boundary. Let $m\geq 2$ be an integer, $r\in \R$ and $r\leq s\in \N$ and $s+1>(n+1)/2$. Assume that $a \in C_c^\infty(\Omega\times \R)$ is admissible and let $\Lambda:H^{s+1}(\Sigma^T)\to H^r(\Sigma^T)$
be the Dirichlet-to-Neumann map of the non-linear wave equation~\eqref{eq:intro_wave-eq}. 
Assume also that $\E:H^{s+1}(\Sigma^T)\to H^r(\Sigma^T)$. 

Let $\eps_0>0$, $M>0$, $0<T<\infty$ and $\delta\in (0,M)$ be such that
 \[
  \norm{\E (f)}_{H^r(\Sigma^T)}\leq \delta, 
 \]
 for all $f\in H^{s+1}(\Sigma^T)$ with $\Vert f\Vert_{H^{s+1}(\Sigma^T)}\leq \eps_0$. 

There are $\tau\geq 1$, $\eps_1,\ldots,\eps_m>0$ and a finite family of functions $\{H_{j}^{\tau,Q}\}\subset H^{s+1}(\Sigma^T)$
  where $j=1,\ldots,m$, and $Q\in \R\times (S^{n-1}\times \R)$, such that 
\begin{align*}
&\sup_{Q\in \R\times (S^{n-1}\times \R)}\Big| 
- \mathscr{R}(a)(Q)  \\
&\qquad \qquad  \qquad  \qquad  - \frac{1}{m!\pi}D_{\eps_1\cdots\eps_m}^m\big|_{\eps=0}\int_{\widetilde{\Sigma}} \psi(\Lambda + \E)(\eps_1H_1^{\tau,Q} +\cdots &+\eps_mH_m^{\tau,Q}) dS 
\Big| \\
& \qquad \qquad\leq C \delta^{2n \sigma (s)}.
\end{align*}
The exponent $\sigma=\sigma(s)$ and the constant $C$ are as in Theorem \ref{thm:stability}. The measurement function $\psi$ is as in \eqref{eq:measurementfun}.
\end{proposition} 

We conclude this section by noting that we expect it to be possible to remove the auxiliary measurement function $\psi$ from the proofs in odd dimensions $n$.
%
The function $v_0$ was essentially only used to accomplish the integration by parts argument, which led to the integral identity~\eqref{eq:integral_identity_finite_difference}. Recall that the potentials $a$ we consider are compactly supported. Therefore, if $n$ is odd,
%
the Huygens principle implies that the terms that depend on $a$ in the expansion of a solution $u$ with respect to the parameters $\eps_1,\ldots,\eps_m$, 
will exit $\Omega$ before time $T$. Thus using $v_0$ is not necessary.
%

%
%
%
 \appendix
 \section{Proofs related to the forward problem}\label{app:A}
 We collect the proofs of the results for the forward problem of Section \ref{sec:forward-problem} here, since they are quite standard.

\begin{proof}[Proof of Corollary \ref{cor:enery-estimate}]
 By Theorem~\ref{thm:KKL-energy} we have that $u\in X^{s+1}$.
 Let $l\in \{1,\ldots,n\}$ and let us denote for simplicity $\p=\p_{x^l}$. Then $\tilde{u}=\p u$ satisfies
 \begin{equation}
\begin{cases}
\square \tilde{u}  = \p F, \quad &\text{ in } \Omega\times[0,T],\\
\tilde{u}=\p f, \,\quad\ &\text{ on } \p\Omega\times[0,T], \\
\tilde{u}\big|_{t=0} = \p \psi_0,\quad \partial_t \tilde{u} \big|_{t=0} = \p \psi_1, &\text{ in } \Omega.
\end{cases}
\end{equation}
Here we have that $\p F\in L^1([0,T]\; H^{s-1}(\Omega))$, $\p \psi_0\in H^{s}(\Omega)$, $\p \psi_1\in H^{s-1}(\Omega)$ and $\p f\in H^{s}(\Sigma)$. We apply Theorem~\ref{thm:KKL-energy} for $\tilde{u}$ and with $s$ replaced by $s-1$. For this, it will be needed that $\p_t^{s-1}\p F\in L^1([0,T],L^2(\Omega))$, which is satisfied by the additional assumption $\p_t^{s-1}F\in L^1([0,T]\; H^{1}(\Omega))$. Therefore, by Theorem~\ref{thm:KKL-energy}, we have that 
\[
 \tilde{u} \in X^{s}=C([0,T]\; H^{s}(\Omega))\cap C^{s}([0,T]\; L^2(\Omega))
\]
and 
\begin{equation}\label{energy_estim_proof_Es}
\begin{split}
 &\norm{\p_t^{s}\tilde{u}(\ccdot,t)}_{L^2(\Omega)} \leq c_T\left(\norm{\p F}_{L^1([0,T]\; H^{s-1}(\Omega))}+\norm{\p_t^{s-1} \p F}_{L^1([0,T]\; L^2(\Omega))}\right.  \\
 &\qquad\qquad\qquad\qquad\qquad\qquad+\left.\norm{\p \psi_0}_{H^{s}(\Omega)}+\norm{\p \psi_1}_{H^{s-1}(\Omega)}+\norm{\p f}_{H^{s}(\Sigma)}\right) \\
 &\qquad\qquad\qquad\quad\leq c_T\big(\norm{F}_{L^1([0,T]\; H^{s}(\Omega))}+\norm{\p_t^{s-1} F}_{L^1([0,T]\; H^1(\Omega))} \\
 &\qquad\qquad\qquad\qquad\qquad\qquad+\norm{\psi_0}_{H^{s+1}(\Omega)}+\norm{ \psi_1}_{H^{s}(\Omega)}+\norm{f}_{H^{s+1}(\Sigma)}\big).
 \end{split}
\end{equation}
Since we also have $u\in X^{s+1}$, the above yields that 
\[
 u \in 
 C^{s}([0,T]\; H^1(\Omega)).
\]

Repeating the argument several times by taking derivatives $\p_{x_{l_1}}\p_{x_{l_2}}\cdots \p_{x_{l_k}}$ of $u$ and the initial and boundary values for all $k=2,\ldots, s$, we obtain
\[
 u \in 
 C^{s+1-k}([0,T]\; H^k(\Omega)), \quad k=0,\ldots,s+1.
\]
Summing up the corresponding estimates similar to~\eqref{energy_estim_proof_Es}, we have~\eqref{energy_estim_Es}. 
\end{proof}

Next we prove Lemma \ref{lemma:nonlinear-solutions} which states that for sufficiently small initial and boundary data there exists a unique small solution to the non-linear wave equation
\begin{equation}\label{eq:proof_of_lemma_banach}
\begin{cases}
\square u + au^m   =0, &\text{in } \Omega\times[0,T],\\
u=\eps f, &\text{on } \p\Omega\times[0,T], \\
u\big|_{t=0} = \partial_t u\big|_{t=0} = 0,&\text{in }\Omega.
\end{cases}
\end{equation}
(This equation is the equation~\eqref{eq:nonlinear_equation_lemma} in Lemma~\ref{lemma:nonlinear-solutions}.)

\begin{proof}[Proof of Lemma \ref{lemma:nonlinear-solutions}]
We prove the existence and uniqueness of small solutions to \eqref{eq:proof_of_lemma_banach} by using the Banach fixed-point theorem. See for example~\cite{Ze95} for the latter.
For this purpose, we define a contraction mapping $\Theta: B_\rho(0) \to B_\rho(0)$ as follows.
Let first $F \in E^{s+1}$.  Then we certainly have 
\[
 \p_t^kF\in L^1([0,T]\; H^{s-k}(\Omega)), \quad k=0,1,\ldots,s.
\]
Assume also that $\p_t^kF|_{t=0}=0$, $k=0,\ldots,s$.
%
%
Let $f\in H^{s+1}(\Sigma)$ be a function satisfying $\p_t^kf|_{t=0}=0$, $k=0,\ldots,s$.
Now, Corollary \ref{cor:enery-estimate} implies that the linear problem
\begin{equation}\label{wave-eq_Banach}
\begin{cases}
\square u  = F, &\text{in } \Omega\times[0,T],\\
u=f, &\text{ on } \p\Omega\times[0,T],\\
u\big|_{t=0} = \partial_t u\big|_{t=0} = 0,&\text{on } \p\Omega
\end{cases}
\end{equation}
%
has a unique solution $u\in E^{s+1}$. Let 
\[
 S: E^{s+1}\cap \{\p_t^kF|_{t=0} =0,\ k=0,\ldots, s\} \to E^{s+1}\cap \{\p_t^kF|_{t=0} =0, \ k=0,\ldots, s\}
\]
be the source-to-solution map which takes $F$ to the corresponding solution $u$ of~\eqref{wave-eq_Banach} (with $f$ fixed) as $F\mapsto u$.
%
%
We then define a new non-linear mapping $\Theta : B_\rho(0) \to B_\rho(0)$ on a ball $B_\rho(0)\subset E^{s+1}\cap \{\p_t^kF|_{t=0} =0,\ k=0,\ldots, s\}$ via the formula
\begin{equation}\label{eq:sigma}
\Theta(u) = S(-au^m),
\end{equation}
where $\rho>0$ will be fixed later. The Banach space $E^{s+1}$ is an algebra since $s+1>(n+1)/2$. We also have $\p_t^k(-au^m)|_{t=0}=0, \ k=0,\ldots,s$. We conclude that the map $\Theta$ is well-defined. 

Let us verify that $\Theta$ defined by \eqref{eq:sigma} is a contraction from a small ball into itself after we have chosen $\kappa$ and $\rho$ small enough.
The energy estimate \eqref{energy_estim_Es} shows that for $u\in B_\rho(0)$ we have 
\begin{align}\label{kappa_dependence}
\Vert \Theta(u) \Vert_{E^{s+1}} &= \Vert S(-au^m) \Vert_{E^{s+1}}
\leq
c_{s,T}\left( \Vert f\Vert_{H^{s+1}(\Sigma)} 
+
\norm{au^m}_{E^s}
 \right)\\
 &\leq 
C_{s,T}\left( \kappa
+ \norm{a}_{C^{s}} \Vert u \Vert_{E^s}^m
\right)\leq C_{s,T}\left( \kappa
+  \norm{a}_{C^{s}} \rho^m\right).
\end{align}
So, if $\rho$ and $\kappa$ are chosen so that
\[
0<\rho^{m-1} < \frac{1}{2C_{s,T} L}
\quad\text{and}\quad
0<\kappa\leq \frac{\rho}{2 C_{s,T}},
\]
then we have that $\Vert \Theta(u)\Vert_ X < \rho$, giving $\Theta : B_\rho(0) \to B_\rho(0)$.

To show that $\Theta$ is a contraction mapping, let $u,v\in B_\rho(0)$. Then the function $S(-au^m) - S(-av^m)$ solves
\begin{equation}
\begin{cases}
\square \big(S(-au^m) - S(-av^m)\big)  = -au^m-av^m, \qquad\text{in } \Omega\times[0,T],\\
S(-au^m) - S(-av^m)=0, \hspace{103pt}\text{on } \p\Omega\times[0,T], \\
\big(S(-au^m) - S(-av^m)\big)\big|_{t=0} = \partial_t \big(S(-au^m) - S(-av^m)\big)\big|_{t=0} = 0, &\text{in } \Omega.
\end{cases}
\end{equation}
Consequently, we have by using the energy estimate~\eqref{energy_estim_Es} again that
%
\begin{align*}
\Vert \Theta(u) - \Theta(v)\Vert_{E^{s+1}} 
&=
\Vert S(-au^m) - S(-av^m) \Vert_{E^{s+1}}\leq
c_{s,T} \norm{au^m-av^m}_{E^{s}} \\
&\leq C_{s,T}\norm{a}_{E^{s}}\norm{u-v}_{E^{s+1}}\norm{P_{m-1}(u,v)}_{E^{s+1}} \\
&\leq \s C_{s,T}m\norm{a}_{E^{s}}\rho^{m-1}\norm{u-v}_{E^{s+1}}. 
\end{align*}
Here we expanded 
    \begin{equation}\label{uv_expansion}
     u^m-v^m=(u-v)P_{m-1}(u,v),
    \end{equation}
where $P_{m-1}(a,b)=\sum_{k=0}^{m-1} a^{m-1-k}b^k$. 
Redefining $\rho>0$ by
\[
 \rho^{m-1}< \frac{1}{2m\s C_{s,T} L}
\]
yields
\[
\Vert \Theta(u)-\Theta(v)\Vert_X \leq \widetilde{C} \Vert u-v\Vert_X.
\]
Here $\widetilde{C}<1/2$. Thus $\Theta : B_\rho(0)\to B_\rho(0)$ is a contraction as claimed.

To finish the proof, note that if the Banach fixed-point iteration is started at $u_0 =0$, we have an estimate for a fixed point $u$ in terms of $u_1 := \Theta(0)$ as follows:
\begin{align*}
\Vert u\Vert_{E^{s+1}} \leq \frac{1}{1-\widetilde{C}} \Vert u_1\Vert_{E^{s+1}}
\leq 2C_{s,T} 
 \Vert f\Vert_{H^{s+1}(\Sigma)}.
\end{align*}
Here we once again used the energy estimate for $u_1$ together with the fact that $\Theta(0)$ corresponds to a solution of the linear problem with no source.
The first inequality follows from a simple argument using geometric sum, see e.g. \cite[Theorem 1.A]{Ze95}. 
\end{proof}
\section{Auxiliary results}\label{app:B}

\subsection{Construction of the measurement function}
Here we construct a measurement function $\psi$ by finding a function $v_0\in C^\infty(\R^{n+1})$ satisfying \eqref{eq:defv_0}. Let us denote
\begin{align*}
\alpha &:= (t_2-t_1 + d)/2,\\
t_0 &:= (t_2+t_1)/2,
\end{align*}
where $t_1,\s t_2$ and $d$ were defined in~\eqref{def:d} and~\eqref{def:t1t2}.
By definition of $d$ there exists $x_0\in \R^n$ such that $\Omega\subset B_{d/2}(x_0)$. Let us pick a smooth cut-off function $\chi_\alpha(l)\in C_0^\infty(\R)$ such that
\begin{equation}
\chi_\alpha(l)=
\begin{cases}
1,&\text{when } |l|\leq\alpha,\\
0,&\text{when } |l|>\alpha+\eps,
\end{cases}
\end{equation}
where $0<\eps<\lambda$ and $\lambda$ is as in \eqref{def:t1t2}.
Let us define
$$
v_0(x,t):= \chi_\alpha((x-x_0)\cdot\theta - (t-t_0))
$$
for some $\theta\in S^{n-1}$ (and $\theta=\pm 1$ when $n=1$).
Clearly $v_0$ satisfies the wave equation $\square v_0 = 0$ in $\R^n\times\R$.

It is also straightforward to verify that $v(x,t) = 1$ if $(x,t)\in \Omega\times[t_1,t_2]$ and $v_0(x,t) = 0$ in $\Omega$, when $t\in [T-r,T]$ for all $0<r<T-\eps$. In particular, $v_0=\p_t v_0 = 0$ in $\Omega\times]T-r,T]$ for all $r\geq 0$ small enough. We may now set $\psi:=v_0|_{\Sigma}$.
%

\subsection{Proofs of lemmata in higher dimensions with $m\geq 2$}

\begin{proof}[Proof of Lemma \ref{RT:pointwise}]
The proof follows from the Fourier slice theorem. We compute the $H^{-\beta}$ norm of $f$ by using polar coordinates on the Fourier side: $\xi \mapsto (\sigma, \theta)$ where $\xi=\sigma\, \theta$ with $\sigma=|\xi|$ and $\theta=\xi/|\xi|$, and so $\d\xi= \sigma^{n-1}\d\sigma\s \d\theta$. 
By using Plancherel and Fourier slice theorems we obtain 
\begin{align*}
(2\pi)^{n-1}\norm{f}^2_{H^{-\beta}(\R^n)}&= (2\pi)^{n-1}\int_{\R^n} (1+ |\xi|^2)^{-\beta}|\widehat{f}(\xi)|^2\s \d\xi\\
&=(2\pi)^{n-1}\int_{S^{n-1}}\int_{\R} (1+\sigma^2)^{-\beta}  |\widehat{f}(\sigma\theta)|^2    \sigma^{n-1} \d\sigma\, \d\theta\\
& = \int_{S^{n-1}}\int_{\R} (1+\sigma^2)^{-\beta}  |\F_{\eta\to\sigma}\big(\pmb{R}f(\theta, \eta)\big)(\sigma)|^2 \sigma^{n-1} \d\sigma\, \d\theta\\
& \leq \int_{S^{n-1}}\int_{\R} (1+\sigma^2)^{\frac{n-1}{2}-\beta}  |\F_{\eta\to\sigma}\big(\pmb{R}f(\theta, \eta)\big)(\sigma)|^2 \d\sigma\, \d\theta\\
& \leq \int_{S^{n-1}}\int_{\R}  |\F_{\eta\to\sigma}\big(\pmb{R}f(\theta, \eta)\big)(\sigma)|^2 \d\sigma\, \d\theta\\
& =(2\pi)^n \int_{S^{n-1}} \int_{\R}  |\pmb{R}f(\theta, \eta)|^2 \d \eta\, \d\theta\\
& \leq (2\pi)^n \, C_0^2\,  \int_{S^{n-1}} \int_{-M}^M |F(\theta, \eta)|^2 \d \eta\, \d\theta\\
&=  (2\pi)^n \, C_0^2\,  \norm{F}_{L^2(S^{n-1}\times [-M,M])}^2.
\end{align*}
Here in the second to last inequality we used that Radon transformation of $f$ is supported in $[-M,M]$ in its  variable $\eta$, since $\text{supp}(f)\subset B_M(0)$. 
\end{proof}

\begin{proof}[Proof of Lemma \ref{estimate:H_jn}]
 Let $\theta\in S^{n-1}$ and $c_0\in \R$. Let us write $c_0=t_0-\eta$ and $F(x,t)= H(x\cdot\theta - t +c_0)$. 
The proof is almost the same as that of Lemma~\ref{est:H_j}. We however sketch a proof to help the reader to see why the constant $C$ is independent of $t_0$, $\theta$ and $\eta$. 

Let $\beta_1$ be a multi-index and $\beta_2\in \N$. Then for all $\tau$ large enough we obtain 
\begin{equation}
 \begin{split}
  &\norm{\p_{x}^{\abs{\beta_1}} \p_t^{\beta_2} F}_{L^2(\Omega\times [0,T])}^2=\tau  \int_{\Omega}\int_0^T\left[\p_{x}^{\abs{\beta_1}} \p_t^{\beta_2}\Big(\chi_\alpha(x\cdot\theta - t + c_0)\, \e^{-\frac{\tau}{2} (x\cdot\theta - t + c_0)^2}\Big)\right]^2 \d t \, \d x\\
&\leq C\s \tau \tau^{2(\abs{\beta_1}+\beta_2)}\int_{\Omega}\int_0^T|\chi_\alpha(x\cdot\theta - t + c_0)|^2\, (x\cdot\theta - t + c_0)^{2(\abs{\beta_1}+\beta_2)}\e^{-\tau (x\cdot\theta - t + c_0)^2} \d t \, \d x\\
&=C\s \tau\tau^{2(\abs{\beta_1}+\beta_2)} \int_{\Omega} \int_{x\cdot\theta+c_0}^{x\cdot \theta -T+c_0} |\chi_\alpha(h)|^2\, h^{2(\abs{\beta_1}+\beta_2)}\e^{-\tau h^2} \d h \, \d x \\
&\leq C\s \tau\tau^{2(\abs{\beta_1}+\beta_2)} \int_{\Omega} \int_{-\infty}^{\infty} |\chi_\alpha(h)|^2\, h^{2(\abs{\beta_1}+\beta_2)}\e^{-\tau h^2} \d h \, \d x \\
&\leq C \tau\tau^{2(\abs{\beta_1}+\beta_2)}\tau^{-(\abs{\beta_1}+\beta_2)-1/2}\int_\Omega dx=C\tau^{\abs{\beta_1}+\beta_2+1/2},
 \end{split}
\end{equation}
where $C$ is independent of $c_0=t_0-\eta$ and $\theta$.
Here we made a change of variables
\[
 h=x\cdot\theta - t +c_0
\]
in the integral in the variable $t$, while considering  $x$ is fixed. We also absorbed terms of lower order powers of $\tau$ into the constant $C$ (see proof of Lemma~\ref{est:H_j} for an explanation),
and used $\int_\R h^{2(\abs{\beta_1}+\beta_2)}e^{-\tau h^2}dh\sim \tau^{-(\abs{\beta_1}+\beta_2)-1/2}$.

By interpolation, see e.g. \cite[Theorem 6.2.4/6.4.5]{BL76}, we obtain for all $\gamma\geq 0$
\begin{equation}\label{est:betas2}
\begin{aligned}
\norm{F}_{H^\gamma(\Omega\times [0,T])}^2& \leq C \tau^{\gamma+1/2}.
\end{aligned}
\end{equation}
By trace theorem, we obtain
\[
\norm{F}_{H^\gamma(\Sigma)}^2\leq C \norm{F}_{H^{\gamma+1/2}(\Omega\times [0,T])}^2\leq  
C\tau^{\gamma + 1}. 
\]
Similar argument yields the same estimate for $H_2^{\tau, (t_0, \theta,\eta)}$. 
This completes the proof.
\end{proof}

\begin{proof}[Proof of Lemma \ref{lemma:final_estimate_n}]
To simplify notation, let $\widehat{s}:=(2m-1)(s+2)/2$ and $\gamma_0=\kappa^{2m-1}/M$. A direct computation shows that
\begin{equation}
\partial_{\epsilon}f= -(\gamma_0\delta) \epsilon^{-m-1}+ \eps^{m-2}\tau^{\hat{s}}, \quad \partial_\tau f=- \tau^{-3/2}+ \widehat{s}\epsilon^{m-1} \tau^{\hat{s}-1}.
\end{equation}
Making $\partial_{\epsilon}f=\partial_{\tau}f=0$, we obtain the critical points of $f$, namely
\begin{align}
\tau&= \left(\frac{\widehat{s}}{m-1}\right)^{-\dfrac{2(2m-1)}{2m(\widehat{s}+1)-1}} (\gamma_0\delta)^{-\dfrac{2(m-1)}{2m(\widehat{s}+1)-1}},\\ 
\eps&=\left(\frac{\widehat{s}}{m-1} \right)^{\dfrac{2\widehat{s}}{2m(\widehat{s}+1)-1}}
(\gamma_0\delta)^{\dfrac{2\widehat{s}+1}{2m(\widehat{s}+1)-1}}.
\end{align}
With these choices of $\tau$ and $\epsilon$, one can check that $\tau^{-1/2}$, $(\gamma_0\delta) \epsilon^{-m}$ and $\epsilon^{m-1} \tau^{\widehat{s}}$ are all bounded by $C_{s,m}\, (\gamma_0\delta)^{(m-1)/[(2m-1)(m(s+2)+1)]}$. 
It is also straightforward to verify that $\tau\geq 1$ for $\kappa$ small enough. 

Furthermore, since
$$
\eps\tau^{\widehat{s}/(2m-1)} = (\gamma_0\delta)^{1/(2m-1)},
$$
we have that
$$
\eps\tau^{\frac{s+2}{2}}
 \leq \kappa
$$
for any $0<\delta<M$.
This finishes the proof. 
\end{proof}

\section{Higher order finite differences}\label{app:C}

Let us define the higher order finite difference operator by
\begin{equation}
D_{\eps_1,\ldots,\eps_m}^m\big|_{\eps=0} u_{\eps_1f_1+\cdots+\eps_mf_m}
=
\frac{1}{\eps_1\cdots\eps_m}\sum_{\sigma\in\{0,1\}^m}
(-1)^{|\sigma|+m}u_{\sigma_1\eps_1 f_1+\ldots+\sigma_m\eps_m f_m},
\end{equation}
where the sum is over all combinations of $\{0,1\}$ of length $m$.
Then for the solution $u$ of \eqref{eq:epsilons} we have
\begin{align}
\square u_{\eps_1f_1+\ldots+\eps_m f_m}
&=
-a\sum_{k_1,\ldots,k_m}\binom{m}{k_1,\ldots,k_m}\eps_{k_1}\cdots\eps_{k_m} v_{k_1}\cdots v_{k_m} + \square \mathcal{R}\\
&=-a(\eps_1 v_1 + \ldots + \eps_m v_m)^m + \square \mathcal{R}.
\end{align}
Applying the finite difference operator to this reduces to the following algebraic identity about numbers.
\begin{lemma}\label{lemma:combinatorics}
Let $x_1,\ldots,x_m\in\R$. Then
\begin{equation}\label{eq:combinatorics}
I(x_1,\ldots,x_m):=\sum_{\sigma\in\{0,1\}^m} (-1)^{m+|\sigma|}(\sigma_1 x_1 +\ldots+\sigma_m x_m)^m
=
m!x_1\cdots x_m.
\end{equation}
\end{lemma}
\begin{proof}
Let $j\in\{1,\ldots,m\}$ and split the summation in \eqref{eq:combinatorics} with respect to
\begin{align*}
\sigma' &= (\sigma_1,\ldots,\sigma_{j-1},1,\sigma_{j+1},\ldots,\sigma_m),\\
\sigma'' &= (\sigma_1,\ldots,\sigma_{j-1},0,\sigma_{j+1},\ldots,\sigma_m).
\end{align*}
Since $|\sigma'| = |\sigma''|+1$, we have
\begin{align*}
&\sum_{\sigma\in\{0,1\}^m} (-1)^{m+|\sigma|}(\sigma_1 x_1 +\ldots+\sigma_j x_j+\ldots+\sigma_m x_m)^m\\
&\quad=
-\sum_{\sigma\in\{0,1\}^{m-1}} (-1)^{m+|\sigma|}(\sigma_1 x_1 +\ldots+1\cdot x_j+\ldots+\sigma_m x_m)^m,\\
&\qquad+
\sum_{\sigma\in\{0,1\}^{m-1}} (-1)^{m+|\sigma|}(\sigma_1 x_1 +\ldots+0\cdot x_j+\ldots+\sigma_m x_m)^m.
\end{align*}
Then note that, if $x_j=0$, the above implies $I(x_1,\ldots,x_{j-1},0,x_{j+1},\ldots,x_m)=0$.
Let us express $I$ via the multinomial formula and write
$$
I(x_1,\ldots,x_m) = I_1+I_1^c,
$$
where $I_1$ contains all terms of $I(x_1,\ldots,x_m)$ of the form $x_1^{p_1}\cdots x_m^{p_m}$ with $p_1\geq 1$ and $I_1^c$ contains the remaining terms of the form $x_1^0x_2^{p_2}\cdots x_m^{p_m}$.
Then, if $x_1=0$, we deduce that $0=I(0,x_2,\ldots,x_m)=I_1^c$.
Next, we split $I_1^c=I_2+I_2^c$, where $I_2$ contains all terms of $I_1$ that have $x_2$ in them, similarly as before.
Then $0=I(x_1,0,x_3,\ldots,x_m)=I_2^c$.
Repeating this process, we deduce that all terms of $I(x_1,\ldots,x_m)$ that miss one of $x_j$, $j=1,\ldots,m$, must cancel.
Hence $I(x_1,\ldots,x_m)=c(m)x_1\cdots x_m$ for a constant $c(m)$.

This term $c(m)x_1\cdots x_m$ only appears in the sum $I$ when $\sigma=(1,\ldots,1)$.
From the multinomial formula we then see that the constant $c(m)=m!$.
\end{proof}

Using Lemma~\ref{lemma:combinatorics} we see that
\begin{equation}
D_{\eps_1,\ldots,\eps_m}^m\big|_{\eps=0} \square u_{\eps_1f_1+\ldots+\eps_m f_m}
=
-m!a v_1\cdots v_m + D_{\eps_1,\ldots,\eps_m}^m \big|_{\eps=0}\square\mathcal{R}.
\end{equation}


\subsection*{Acknowledgements}

M. L. was supported by Academy of Finland, grants 320113, 318990, and 312119. L. P-M. and T. L. were supported by the Academy of Finland (Centre of Excellence in Inverse Modeling and Imaging, grant numbers 284715 and 309963) and by the European Research Council under Horizon 2020 (ERC CoG 770924). T. T  was supported by the Academy of Finland (Centre of Excellence in Inverse Modeling and Imaging, grant number 312119).

\noindent{\footnotesize E-mail addresses:\\
Matti Lassas: {matti.lassas@helsinki.fi}\\
Tony Liimatainen: {tony.liimatainen@helsinki.fi}\\
Leyter Potenciano-Machado: {leyter.m.potenciano@jyu.fi}\\
Teemu Tyni: {teemu.tyni@helsinki.fi}
}

\end{document}